\documentclass[a4paper]{amsart}

\usepackage{amssymb}
\usepackage{amsmath,amsthm}
\usepackage{enumerate,paralist}
\usepackage{amstext}
\usepackage{psfrag}
\usepackage{dsfont}
\usepackage[dvips]{epsfig}
\usepackage[english]{babel}
\usepackage{cite}
\usepackage{hyperref}
\usepackage{color}
\usepackage{cleveref}
\usepackage[svgnames]{xcolor}
\hypersetup{hidelinks,colorlinks=true,allcolors=DarkBlue}

\usepackage{cite}

\theoremstyle{plain}
\newtheorem{theorem}{Theorem}
\newtheorem{corollary}{Corollary}
\newtheorem{lemma}{Lemma}
\theoremstyle{definition}

\def\diag{{\rm diag}}
\def\sign{\,{\rm sign}}
\def\sing{{\rm sing}}

\def\div{{\rm div}}

\begin{document}
\frenchspacing

\title[Asymptotic control theory for a system of linear oscillators]
{Asymptotic control theory \\ for a system of linear oscillators}

\author{Aleksey Fedorov}
\address{
Institute for Problems in Mechanics, Russian Academy of Sciences \\
119526, Vernadsky av., 101/1, Moscow, Russia \&
Russian Quantum Center \\
143025 Novaya st. 100, Skolkovo, Moscow, Russia \&
Laboratoire de Physique Th{\'e}orique et Mod{\`e}les Statistiques, CNRS and Universit\'e Paris Sud, UMR8626, 91405 Orsay, France}
\email{akf@rqc.ru}

\author{Alexander Ovseevich}
\address{Institute for Problems in Mechanics, Russian Academy of Sciences \\
119526, Vernadsky av., 101/1, Moscow, Russia.} 
\email{ovseev@ipmnet.ru}

\maketitle

\begin{abstract}
We present an asymptotic control theory for a system of an arbitrary number of linear oscillators under a common bounded
control. We suggest a design method of a feedback control for this system. By using the DiPerna--Lions theory of singular
ODEs, we prove that the suggested control law correctly defines the motion of the system. The obtained control is
asymptotically optimal: the ratio of the motion time to zero under this control to the minimum one is close to 1 if the
initial energy of the system is large. The results are partially based on a new perturbation theory of  observable linear
systems.


\end{abstract}

\tableofcontents

\section{Introduction}
The problem of time-optimal steering of a given initial state to a given manifold is typical for optimal control theory.
One of the classical achievements in this area is an explicit construction of the minimum time damping of a single linear
oscillator \cite{pont}. This system is governed by the equation
\begin{equation}
    \ddot{x}+x=u, \quad |u|\leq1,
\end{equation}
where $x$ is the position and $u$ is the control. Here the oscillator frequency $\omega$ is assumed to be 1. The optimal
control is of bang-bang type, i.e., it takes values $u=\pm1$. The switching curve, which separates the $u=-1$ domain of
the phase plane from the $u=+1$ domain, consists of unit semicircles centered at points of the form $(2k+1,0)$, where
$k\in\mathbb{Z}$ is an integer. When seen from afar, which is our primary point of view in this paper, the switching
curve looks like the $x$-axis, and the optimal control looks like the dry friction $u(x,\dot x)=-\sign{\dot x}$.

\subsection{Problem statement}
This paper is devoted to a more general and next in complexity problem of minimum-time steering of a system of $N$ linear
oscillators with eigenfrequencies $\omega_i$ under a common bounded control $u$ described by
\begin{equation}\label{oscillators}
\begin{array}{l}
    \dot{x_i}=y_i \\
    \dot{y_i}=-\omega_i^{2}{x_i}+{u}, \quad |u|\leq1, \quad i=1,\dots,N.
\end{array}
\end{equation}
It is probably impossible in principle to obtain an explicit formula for the optimal control in this case. Even a
numerical solution appears difficult.

In this paper, we deal with feedback control and  try to make the duration of the steering as small as possible. We
assume that the steering is possible in principle. This means, according to the Kalman controllability condition
\cite{kalman}, that the eigenfrequencies $\omega_i$ of all oscillators are different.

Our main result is design of an asymptotically optimal and numerically implementable feedback control for system
(\ref{oscillators}) in the non-resonant case, when there are no nontrivial relations between eigenfrequencies of the form
\begin{equation}
\label{resonance}
    \sum_{i=1}^N m_i\omega_i=0, \mbox{ where } 0\neq m=(m_1,\dots,m_N)\in {\mathbb{Z}}^N.
\end{equation}
Here `asymptotic' refers to the large initial energy
\begin{equation}
    E=\frac{1}{2}\sum_{i=1}^{N}{(\dot{x_i}^2+\omega_i^2 x_i^2)}
\end{equation}
of system (\ref{oscillators}). Our control works in the resonant case as well but in that case is not asymptotically
optimal. Still, the ratio of the steering time to the minimum time is uniformly bounded.

System (\ref{oscillators}) can be interpreted in mechanical terms in at least two ways. In the first model, the
components $x_i$ of the state vector are vertical deviations of pendulums attached to a cart moving with bounded
acceleration $u$. In the second model, the components $x_i$ are displacements of the masses attached to the springs
attached to the same cart.

\subsection{Minimum time problem}\label{min_time}
Suppose we need to bring system (\ref{oscillators}) to equilibrium in minimum time. This problem is a particular case of
the minimum time problem for a linear control system
\begin{equation}\label{syst1}
    \dot{x}={A}x+{B}u,
    \quad
    x=(x_1,y_1,\dots,x_N,y_N)^*\in\mathbb{V}={\mathbb{R}}^{2N},
    \quad
    u\in\mathbb{U}={\mathbb{R}},
    \quad
    |u|\leq1,
\end{equation}
where the matrix ${A}$ and the vector ${B}$ are
\begin{equation}\label{syst2}
    {A} = \left( {\begin{array}{*{20}c}
    0           & 1 & &                      &   \\
    { - \omega_1^{2}} & 0& &                        &  \\
    &  & \ddots&                         &   \\
    &  & &0                      & 1  \\
    &  & &{-\omega_N^{2}} & 0  \\
 \end{array} } \right), \qquad
{B} = \left( \begin{gathered}
    0 \hfill \\
    1 \hfill \\
    \vdots \hfill \\
    0 \hfill \\
    1 \hfill \\
\end{gathered}\right).
\end{equation}
As it is well known, the problem is equivalent to the boundary value problem for the Pontryagin maximum principle
corresponding to the Hamiltonian
\begin{equation}
    h(x,\psi)=\langle{Ax,\psi}\rangle+|\langle{B,\psi}\rangle|-1=\max_{|u|\leq1}\{\langle{Ax,\psi}\rangle+\langle{Bu,\psi}\rangle-1\}.
\end{equation}
Here angle brackets denote the standard scalar product in $\mathbb{R}^{2N}$, $|\cdot|$ is the Euclidean norm, and the
maximum is taken over the interval $\{u\in\mathbb{R}:|u|\leq1\}$.

The problem takes the form
\begin{equation}\label{max}
\begin{array}{l}
    \dot{x}={A}x+{B}u, \quad \dot{\psi}=-{A^*}\psi,\\[.5em]
    u=\sign\langle{B,\psi}\rangle, \quad x(0)=x_0, \quad x(T)=0, \quad h(x,\psi)=0.
\end{array}
\end{equation}
We note that system (\ref{max}) is Hamiltonian with $2N$ degrees of freedom and $N+1$ integrals of motion. In order to
define these integrals we use the following notations. Suppose that the momentum $\psi$ is written in the form
$\psi=(\psi_i)$, where $\psi_i=(\xi_i,\eta_i)$, ${i=1,\dots,N}$, $\xi_i$ is the dual variable for $x_i$ and $\eta_i$ is
the dual variable for $y_i$. There are $N$ integrals of the form
\begin{equation}\label{integrals_motion}
I_k=\frac12\xi_k^2+\frac12\omega_k^{2}\eta_k^2, \quad k=1,\dots,N,
\end{equation}
and the Hamiltonian $h(x,\psi)$. These integrals are Poisson commuting. The fact that the Poisson brackets $\{I_k,I_l\}$
are zero is obvious, while the identity $\{I_k,h\}=0$ results from an easy computation. In the case $N=1$, the number of
degrees of freedoms coincides with the number of commuting integrals. This is the basic reason for the existence of an
explicit optimal solution. The same equality is the basic assumption of the Liouville--Arnold theorem on complete
integrability of a Hamiltonian system \cite{Arnold}.

In general, we deal with a nonlinear boundary value problem of dimension $4N$. If the vector $\psi(0)$ is known, then the
control $u$ is also known, and $x(T)$ can be easily found via solution of the Cauchy problem. Therefore, the boundary
problem reduces to a solution of $2N+1$ transcendental equations $x(T)=0$ and $h(x,\psi)=0$ for the $2N+1$-dimensional
vector with components $\psi(0),\,T$. The great difficulty of this problem suggests to resort to approximations.

\subsection{Proposed strategy and main results}\label{suggest}
We present a method based on consideration of the asymptotic behavior of reachable sets of the system. We use both
$T\to\infty$ asymptotic and $T\to0$ asymptotic of reachable sets.

Let us divide the entire phase space of the system into three zones: the high-energy zone, the middle-energy zone, and a
small neighborhood of the equilibrium, i.e., the low-energy zone. In the high-energy zone, we use a control law based on
an asymptotic $T\to\infty$ formula for the support function of reachable sets \cite{ovseev,ovseev2} for system
(\ref{oscillators}). The control can be in principle applied as well in other zones, but then its quasi-optimal
properties are lost. Moreover, the control affects the system like dry friction, so that in some states, where the energy
is not too high, it prevents any motion. More generally, the control might force the system to move in a small
neighborhood of a limit set (attractor) not containing the target, i.e., the equilibrium state. In other words, there
arise basins of attraction; the greater is the upper bound for controls, the larger are the basins.

To prevent getting into an attractor, we use within the middle energy zone a scaled version of the high-energy control
with a reduced amplitude. This makes the basins of attraction located in a smaller neighborhood of the target, so that
the sinking into an attractor cannot happen within high and middle energy zones. This strategy allows the system to reach
a small neighborhood of the equilibrium, where a terminal control scenario is in force.

In the third terminal stage, we consider the asymptotic behavior of reachable sets as $T\to0$. We use important
properties of shapes of the reachable sets: by applying gauge transformations and adding a linear feedback we do not
substantially change the shapes of the reachable sets \cite{ovseev_gonch}. By using these properties, we reduce the
problem of the feedback control design for system (\ref{syst1}) to the design of a feedback control for a canonical
system in Brunovsky form \cite{brun}. Toward this end, we apply a method of control based on common Lyapunov functions
\cite{korobov,korob,anan}.

Let us stress that  our main goal is asymptotic optimality. Therefore, the detailed construction of the control within
finite distance to the equilibrium is of secondary importance.

Our paper is organized as follows. In Section \ref{high}, we define the control and describe its implementation in
detail. In Section \ref{property}, we discuss formal properties of the control within high and medium energy zones. In
particular, we establish a maximum principle of a certain kind for the suggested control. A nontrivial issue of the
nature of the dynamics of the system is discussed in Section \ref{motion}. We prove the existence and uniqueness of the
motion under the control within the framework of the DiPerna--Lions theory. Asymptotic optimality of the control within
high-energy zone is proved in Section \ref{quasioptimal}. In Section \ref{efficiency}, we study efficiency of the
suggested control by using a new technique based on perturbation theory of observable linear systems (see Appendix
\ref{perturb_observ}). Section \ref{attractor} is devoted to the singular arcs of our control. We find the size of a ball
centered at the equilibrium which does not contain any attractor. Section \ref{terminal} describes the design of the
feedback control at the final stage, i.e., in a small neighborhood of the target. We utilize the common Lyapunov
functions technique and demonstrate a few nontrivial features of its application including those of number-theoretical
nature. In Section \ref{match}, we perform the matching of controls defined within different zones. Our main result on
asymptotic optimality is presented in Section \ref{asymp}. In Section \ref{toy}, we illustrate our strategy in the
classical case of a single oscillator. Appendices \ref{support}-\ref{proof_3} contain a number of auxiliary results.

A summary of our results was presented in \cite{ovseev4}.

\section{Basic control: high-energy zone}\label{high}
A well-known geometric interpretation of the maximum principle says that the momentum (adjoint vector) $\psi$ at point
$x$ is the inner normal to the reachable set $\mathcal{D}(T(x))$ \cite{agrachev}.

Here the reachable set $\mathcal{D}(T)$ is the set of ends at time instant $T$ of all admissible trajectories of system
(\ref{syst1})--(\ref{syst2}) starting at the origin at zero time.

\subsection{Asymptotic theory of reachable sets as $T\to\infty$}
We would like to use as momenta the normals to an approximate reachable set. This is possible thanks to the asymptotic
theory of reachable sets for linear systems as developed in \cite{ovseev}.

One of the basic results of Ref. \cite{ovseev} applicable to our system of $N$ oscillators is this: The reachable set
$\mathcal{D}(T)$ equals asymptotically as $T\to\infty$ to the set  $T\Omega$, where $\Omega$ is a fixed convex body. More
precisely:
\begin{theorem}\label{support0}{\rm \cite{ovseev}}
    Suppose that a momentum $p$ is written in the form $p=(p_i)$,
    where $p_i=(\xi_i,\eta_i)$, ${i=1,\dots,N}$,
    $\xi_i$ is the dual variable for $x_i$,
    $\eta_i$ is the dual variable for $y_i$, and
    $z_i=(\eta_i^2+{\omega_i^{-2}}{\xi_i^{2}})^{1/2}$.
    Suppose that system (\ref{syst1})--(\ref{syst2}) is non-resonant, i.e., there are no nontrivial relations (\ref{resonance}).
    Then, the support function $H_T$ of the reachable set $\mathcal{D}(T)$ has as $T\to\infty$ the asymptotic form
    \begin{equation}\label{approxN0}
        {H}_T(p)=T\int\limits_{\mathcal{T}}\left|\sum_{i=1}^N z_{i}\cos\varphi_{i}\right|d\varphi+o(T),
        \quad
        d\varphi=\frac{1}{(2\pi)^N}d\varphi_{1}\wedge\dots\wedge d\varphi_{N},
    \end{equation}
\end{theorem}
We give a sketch of the proof of Theorem \ref{approxN0} in Appendix \ref{support}.

Recall that the support function of any subset $M\subset{\mathbb{R}}^{n}$ is defined as ${H}_{M}(\xi)=\sup_{x\in
M}\langle{\xi, x}\rangle$ and defines the closed convex hull of $M$ uniquely \cite{Schneider}. In particular, the support
function of the convex body $\Omega$ is given by the main term in (\ref{approxN0}):
\begin{equation}\label{approx2N}
    {H}_\Omega(p)=\mathfrak{H}(z)=\int\limits_{\mathcal{T}}\left|\sum_{i=1}^N z_{i}\cos\varphi_{i}\right|d\varphi,
\end{equation} { where the vector}
$z=(z_1,\dots,z_N)\in\mathbb{R}^N$ has components $z_i=(\eta_i^2+{\omega_i^{-2}}{\xi_i^{2}})^{1/2}$.

If $N=1$, we obtain $\mathfrak{H}(z)=\frac2\pi|z|$. In the case $N=2$, the function
\begin{equation}
    \mathfrak{H}(z)=\int\left|z_{1}\cos\varphi_{1}+z_{2}\cos\varphi_{2}\right|d\varphi
\end{equation}
can be expressed via elliptic integrals as follows:
\begin{equation}\label{elliptic_integral}
    \mathfrak{H}(z_1,z_2)=\frac{1}{\pi^2}\int_0^{2\pi}\frac{(z_2^2-z_1^2)d\varphi}{\sqrt{z_2^2-z_1^2\cos^2\varphi}}
    \mbox{ if }
    |z_1|\leq|z_2|.
\end{equation}
(see Appendix \ref{apelliptic}). In general, by substitution $t_i=\cos\varphi_i$ we reduce (\ref{approx2N}) to an
Euler-type integral
\begin{equation}
    \mathfrak{H}(z)=\frac{1}{(2\pi)^N}\int\limits_{\{|t_i|\leq1\}}{\left|\sum_{i=1}^N z_{i}t_{i}\right|}{\prod_{i=1}^N(1-t_i^2)^{-1/2}}dt_1\dots dt_N
\end{equation}
that defines a hypergeometric function in the sense of I.M.~Gelfand \cite{Varchenko}. The function $\mathfrak{H}(z)$ also
has an (one-dimensional) integral representation via the Bessel functions (see Appendix \ref{gothic_H}).

Note that equation (\ref{approx2N}) makes sense even in the resonant case, when there are nontrivial relations between
eigen\-fre\-quen\-cies. In this case however, equation (\ref{approxN0}) does not give an asymptotic formula for the
support function of the reachable set $\mathcal{D}(T)$.

The basic idea of our feedback control is to substitute the set $T\Omega$ for $\mathcal{D}(T)$. The idea works even in
the resonant case, when $T\Omega$ is not an asymptotic approximation of $\mathcal{D}(T)$. Note that a phase vector
$x\in\mathbb{V}={\mathbb{R}}^{2N}$ belongs to the boundary of $T\Omega$ if and only if
\begin{equation}\label{approx3}
    T^{-1}x=\frac{\partial {H}_\Omega}{\partial p}(p)
\end{equation}
for a momentum $p=p(x)$. We notice that the support function ${H}_{\Omega}$ is differentiable, and equation
(\ref{approx3}) has a unique solution up to scaling $p\mapsto\lambda p,\,\lambda>0$, because the boundary of $\Omega$ is
smooth \cite{ovseev2}. The unique solvability of equation (\ref{approx3}) is also proved below in Section
\ref{KuhnTucker}. We discuss the issue of efficient solution of equation (\ref{approx3}) in the next section.

Thus, our {\em basic control} in the high-energy zone is given by
\begin{equation}\label{approx4}
    u(x)=-\sign\langle{B,p(x)}\rangle,
\end{equation}
and it depends on the direction of the vector $p(x)$ only, so that the scaling $p\mapsto\lambda p,$ where $\lambda>0$,
does not affect the control. We emphasize that the minus sign in (\ref{approx4}) is due to the fact that $p(x)$ is the
{\it outer} normal to $T\Omega$ at the point $x$, while the momentum $\psi$ in the Pontryagin maximum principle is the
{\it inner} normal to the exact reachable set.

\subsection{Efficient computation of the control}
In coordinates $x_i,y_i$, equation (\ref{approx3}) takes the form
\begin{equation}\label{approx_coord}
    T^{-1}(x_i,y_i)=z_i^{-1}\left(\frac{\partial {\mathfrak H}}{\partial z_i}\right)\left(\frac{\xi_i}{\omega_i^2},\eta_i\right),\quad i=1,\dots,N,
\end{equation}
where $z_i=(\eta_i^2+{\omega_i^{-2}}{\xi_i^{2}})^{1/2}$, and ${\mathfrak H}(z)$ is given by integral (\ref{approx2N}). To
solve (\ref{approx_coord}), we should first find the point $\mathfrak z$ of the sphere $S^{N-1}$ with
positive-homogeneous coordinates $(z_1:\dots:z_N)$. Here the sphere $S^{N-1}$ is regarded as the set of directions of
non-zero vectors in $\mathbb{R}^N$. To this end, we define the ``energetic'' vector $e=(e_i)\in\mathbb{R}^N,$ where
$e_i=(\omega_i^{2}x_i^{2}+y_i^{2})^{1/2}$, and obtain from (\ref{approx_coord}) that
\begin{equation}\label{z}
    T^{-1}e_i=\frac{\partial {\mathfrak H}}{\partial z_i}(\mathfrak z),\quad i=1,\dots,N.
\end{equation}
Solution of equation (\ref{approx_coord}) gives an inversion of a map from one $2N$-dimensional manifold to another,
while the solution of (\ref{z}) reduces to inversion of a map of $(N-1)$-dimensional manifolds. Still the solution of
(\ref{approx_coord}) reduces easily to the solution of (\ref{z}). Similarly to the master equation (\ref{approx3}), equation
(\ref{z}) has, according to \cite{ovseev2}, a unique solution, which, however, is not a very easy find. Anyway, we obtain
that $T$ is a function of the ``energetic'' vector $e$.

\subsection{Kuhn--Tucker theorem}\label{KuhnTucker}
The Kuhn--Tucker theorem implies that for arbitrary $N$ the search for solutions of (\ref{z}) is equivalent to the
optimization problem
\begin{equation}\label{optim}
    \langle e,z\rangle\to\max \mbox{, provided that }\mathfrak H(z)\leq1,
\end{equation}
and similar approach can be applied to equations (\ref{approx3}) and (\ref{approx_coord}).

It is clear that the constraint $\mathfrak H(z)\leq1$ is equivalent to $\mathfrak H(z)=1$. 
The hypersurface 
$\{\mathfrak H(z)=1\}$ is strictly convex because of the obvious identity
\begin{equation}\label{hessian}
    \left\langle\frac{\partial^2 {\mathfrak H}}{\partial z^2}(z)\xi,\xi\right\rangle=\int\limits_{V(z)}\left(\sum_{i=1}^N\xi_{i}\cos\varphi_{i}\right)^2d\sigma(\varphi),
\end{equation}
where integration is over $V(z)=\left\{\varphi\in\mathcal{T}:f(z,\varphi)=0\right\}$,
\begin{equation}
    f(z,\varphi)=\sum_{i=1}^N z_{i}\cos\varphi_{i},\quad d\sigma(\varphi)=\frac{d\varphi_1\wedge\dots\wedge d\varphi_N}{(2\pi)^N df}
\end{equation}
is the canonical volume element on $V(z)$. Identity (\ref{optim}) implies that, if the vectors $\xi$ and $z$ are not
collinear, then $\left\langle\frac{\partial^2 {\mathfrak H}}{\partial z^2}(z)\xi,\xi\right\rangle$ is strictly positive.
But if the vector $\xi$ is tangent to the hypersurface $\{\mathfrak H(z)=1\}$ at $z$, these two vectors cannot be
collinear. Otherwise, we would obtain that $\left\langle{{\partial {\mathfrak H}}/{\partial z},z}\right\rangle=0$, which
is impossible, since $\left\langle{{\partial{\mathfrak H}}/{\partial z},z}\right\rangle={\mathfrak H}(z)>0$ in view of
the Euler identity. The proved strict convexity of $\{\mathfrak H(z)=1\}$, as it is well-known, implies the uniqueness of
solution of optimization problem (\ref{optim}). Indeed, it follows from the strict convexity of $\{\mathfrak H(z)=1\}$
that the function $f={\mathfrak H}^2$ is strictly convex. At the same time, optimization problem (\ref{optim}) is
equivalent to
\begin{equation}\label{optim2}
    \langle e,z\rangle\to\max \mbox{, provided that } f(z)\leq1.
\end{equation}
If $z_1\neq z_2$ are  solutions to (\ref{optim2}), then 
\begin{equation}
\langle{e,z_1}\rangle=\langle{e,z_2}\rangle \mbox{ and } f(z_i)=1.
\end{equation}
However, this implies that 
\begin{equation}
\left\langle{e,\frac{z_1+z_2}{2}}\right\rangle{=}\langle{e,z_1}\rangle \mbox{ and } f\left(\frac{z_1+z_2}{2}\right)<1,
\end{equation}
which contradicts optimality of $z_i$.

Thus, optimization problem (\ref{optim}) can be solved by well-developed efficient methods, which are still more
difficult than the solution of a scalar transcendental equation. These methods are available, e.g., via Matlab
Optimization Toolbox.

We now obtain from (\ref{z}) the final formula for the momentum:
\begin{equation}\label{impulse}
    ({\xi_i},\eta_i)=\frac{z_i}{e_i}({\omega_i^2}x_i,y_i),\, i=1,\dots,N.
\end{equation}
Thus, if we know the point $\mathfrak {z}=(z_1:\dots:z_N)\in S^{N-1}$, then the direction of the momentum $p(x)$ is
defined by (\ref{impulse}) uniquely. Control (\ref{approx4}) depends only on the direction of the momentum. Therefore, it
can be efficiently found in the form
\begin{equation}\label{control}
    u(x)=-\sign\left(\sum_{i=1}^N{e_i}^{-1}{z_iy_i}\right).
\end{equation}
In the case $N=1$, the control has the form of a dry friction $u=-\sign{y_1}$.

The $\sign$-function in (\ref{control}) is understood as a multivalued map: $\sign(x)=\pm1$ if $x\gtrless0,$ and
$\sign(0)$ might take any value from the interval $[-1,1]$. The precise value of the control in the case of indefinite
sign is sometimes important (see Section \ref{attractor11}). Whatever the precise value is, the control $u(x)$ is not a
continuous function of $x$. Therefore, to define the motion under the control we have to solve ODE with a discontinuous
right-hand side (RHS). This naturally requires a discussion of singular ODEs, which we provide in Section \ref{motion}. In
what follows, we will also use a scaled control $u_U(x)=Uu(x)$ with a smaller amplitude $|U|\leq1$.

\section{Formal properties of the basic control}\label{property}

\subsection{Polar-like coordinate system}
We define a polar-like coordinate system, well suited for repre\-sen\-ta\-tion of the motion under the control $u$. If
$N=1$, we get the proper polar coordinate system in a plane. To this end, we take the boundary $\omega=\partial\Omega$ of
the set $\Omega$ with support function (\ref{approx2N}) as a unit ``sphere''. Every vector $0\neq x\in\mathbb{R}^{2N}$
can be represented uniquely as
\begin{equation}\label{rhophi}
    x=\rho\phi,\mbox{ where }\rho=\rho(x)\mbox{ is a positive factor, and }\phi\in\omega.
\end{equation}
In fact, we have already familiar with (\ref{rhophi}) because  (\ref{approx3}) says exactly the same if $\rho=T$, and
$\phi={\partial H_\Omega}/{\partial p}$. The pair $\rho,\phi$ is the coordinate representation for $x$, and
$\rho(\phi)=1$ is the equation of the ``sphere'' $\omega$. It is important that the set $\omega$ is invariant under free
(uncontrolled) motion of our system (\ref{syst1}). This follows from the similar invariance of the support function
${H}_{\Omega}(p)$ under evolution governed by $\dot{p}=-{A^*}p$. The latter invariance is clear, because the support
function depends only on variables $z_i$, which are integrals of the motion. The invariance of $\omega$ is equivalent to
invariance of the homogeneous function $\rho,$ so that $\left\langle{\partial {\rho}}/{\partial x},Ax\right\rangle=0$.
Therefore, under the control $u$ the total (Lie) derivative of $\rho$ takes the form
\begin{equation}\label{T}
    \dot \rho=\left\langle{\frac{\partial {\rho}}{\partial x},Ax+Bu}\right\rangle=
    \left\langle{\frac{\partial {\rho}}{\partial x},Bu}\right\rangle=-\left|\left\langle{\frac{\partial {\rho}}{\partial x},B}\right\rangle\right|,
\end{equation}
where the last identity holds because ${\partial {\rho}}/{\partial x}$ is the outer normal to the set $\rho\Omega$. Note
that the ``radius'' $\rho$ is monotone nonincreasing. For any other admissible control, we have
\begin{equation}\label{T3}
    \dot\rho\geq-\left|\left\langle{\frac{\partial {\rho}}{\partial x},B}\right\rangle\right|.
\end{equation}
The evolution of $\phi$ by virtue of system (\ref{syst1}) is described by
\begin{equation}\label{xi}
    \dot \phi=A\phi+\frac{1}{\rho}(Bu-\phi \dot \rho)
    =A\phi+\frac{1}{\rho}\left(Bu+\phi \left|\left\langle{\frac{\partial{\rho}}{\partial x},B}\right\rangle\right|\right).
\end{equation}
It is clear that if  $\rho$ is large, then the second term in the RHS of (\ref{xi}) is $O(1/\rho)$ and affects the motion
of $\phi$ over the ``sphere'' $\omega$ only slightly. The conclusion holds for any admissible control, not just for
control (\ref{approx4}).

We note that  $p={\partial {\rho}}/{\partial x}$ is a homogeneous function of degree 0, and, therefore, is a function of
$\phi$. Geometrically speaking, $p$ is the outer normal to the surface $\omega$ at  $\phi$. It follows immediately from
the Euler identity that
\begin{equation}\label{Euler}
    H_\Omega\left(p\right)=\langle{p,\phi}\rangle=\rho(\phi)=1.
\end{equation}
Thus, the function $\rho$ satisfies an eikonal-type equation which is ``dual'' to  equation
$\rho(\frac{\partial{H}}{\partial p})=1$ of the surface $\omega$. Here $H$ stands for $H_\Omega$. We will use Eq.
(\ref{Euler}) in Section \ref{quasioptimal} for averaging the RHS of identity (\ref{T}) with respect to time.

\subsection{Duality transform}
Here we discuss a general duality transformation related to equation (\ref{approx3}). Toward this end we denote the
function $H_\Omega$ just by $H=H(p)$, and the factor $T$ by $\rho(x)$. Then the relation between $H$ and $\rho$ is
similar to the Legendre transformation:
\begin{equation}\label{Legendre}
    \langle x,p\rangle=\rho(x) H(p),\quad \rho(x)=\max_{H(p)\leq1}\langle x,p\rangle,
    \quad
    H(p)=\max_{\rho(x)\leq1}\langle{x,p}\rangle,
\end{equation}
where the correspondence  $x\rightleftarrows p$ has the form
\begin{equation}\label{Legendre2}
    x=\rho(x)\frac{\partial H}{\partial p}(p),\quad p=H(p)\frac{\partial \rho}{\partial x}(x).
\end{equation}
Here $p$ and $x$ are the points where the maximums in (\ref{Legendre}) are attained. The eikonal-type equation
(\ref{Euler}) also holds in this generality. Indeed, by inserting (\ref{Legendre2}) into (\ref{Legendre}) we obtain
\begin{equation*}\label{Legendre3}
    H\left(\frac{\partial \rho}{\partial x}\right)\rho(x) H(p)=\langle x,p\rangle,
    \quad
    \rho\left(\frac{\partial H}{\partial p}\right)\rho(x) H(p)=\langle x,p\rangle,
\end{equation*}
which implies
\begin{equation}\label{Legendre4}
    H\left(\frac{\partial \rho}{\partial x}\right)=1,\quad \rho\left(\frac{\partial H}{\partial p}\right)=1.
\end{equation}

These relations make sense provided that $H$ and $\rho$ are norms, i.e., the homogeneous of degree 1 convex functions
such that the sublevel sets $\{H(p)\leq1\}$ and $\{\rho(x)\leq1\}$ are convex bodies. These sublevels are mutually polar
to each other. In other words, if $\Omega=\{\rho(x)\leq1\}$, and ${\Omega}^{\circ}=\{H(p)\leq1\}$, then
$\Omega=\{x:\langle{x,p}\rangle\leq1,\,p\in{\Omega}^\circ\}$ and vice versa. In the language of the Banach spaces, the
normed spaces $(\mathbb{V},\rho)$ and $(\mathbb{V}^*,H)$ are dual to each other. The derivatives in (\ref{Legendre2})
should be understood as subgradients. If the functions $H$ and $\rho$ are differentiable, equation (\ref{Legendre2}) has
the classical meaning. If one of the functions $H$ and $\rho$ is differentiable and strictly convex, then the other one
is also so.

We notice that apart from the dual pair $(H,\rho)$ there is another related natural dual pair $(\mathfrak H,\mathfrak
R)$, where $H(p)=\mathfrak H(z(p))$, $\rho(x)=\mathfrak R(e(x))$. Here $z(p)=(z_i)$ is the $N$-vector with components
$z_i=(\eta_i^2+{\omega_i^{-2}}{\xi_i^{2}})^{1/2}$, and $e(x)=(e_i)=((\omega_i^{2}x_i^{2}+y_i^{2})^{1/2})$.

From defining relation (\ref{Legendre2}) with $\rho={\partial p}/{\partial x}$, we obtain
\begin{equation}\label{rho0}
    x=\rho\frac{\partial {H}}{\partial p}\left(\frac{\partial\rho}{\partial x}\right).
\end{equation}
Differentiating (\ref{rho0}), we obtain a relation between the second derivatives of the dual functions
\begin{equation}\label{rho02}
    1=\rho\frac{\partial^2 {H}}{\partial p^2}\frac{\partial^2\rho}{\partial x^2}+\frac{\partial\rho}{\partial x}\otimes\frac{\partial {H}}{\partial p},
\end{equation}
or, using more detailed notation, that for any constant vector $\zeta$
\begin{equation}\label{rho03}
    \zeta=\rho\frac{\partial^2 {H}}{\partial p^2}\frac{\partial^2\rho}{\partial x^2}\zeta+\left\langle{\frac{\partial\rho}{\partial x},\zeta}\right\rangle\frac{\partial {H}}{\partial p}.
\end{equation}
Differentiability of functions $H,\rho,\mathfrak{H}$, and $\mathfrak{R}$ is studied in Appendix \ref{rho}.

\subsection{Hamiltonian structure}\label{hamilton}
Here we show that basic control (\ref{approx4}) possesses a Hamiltonian structure. This means that we can extend the
corresponding dynamical system to a canonical one, similar to that of maximum principle (\ref{max}). This requires
understanding the time-evolution of the momentum $p(x)$ involved in (\ref{approx4}). We possess the expression
$p=\frac{\partial\rho}{\partial x}(\phi)$ for the momentum, where the point $\phi$ makes a controlled motion satisfying
(\ref{xi}). It follows from the identity
\begin{equation}\label{hamilton01}
    \left\langle{\frac{\partial\rho}{\partial x},A\phi}\right\rangle=0,
\end{equation}
which expresses the invariance of the ``radius'' under free motion, that for any (constant) vector $\zeta$ we have
\begin{equation}
    \left\langle{\frac{\partial^2\rho}{\partial x^2}\zeta,A\phi}\right\rangle=-\left\langle{\frac{\partial\rho}{\partial x},A\zeta}\right\rangle.
\end{equation}
On the other hand, the total derivative $\langle{\dot p,\zeta}\rangle$ can be written in the form
\begin{equation}
    \left\langle{\frac{\partial^2\rho}{\partial x^2}(A\phi+Bu),\zeta}\right\rangle=\left\langle{\frac{\partial^2\rho}{\partial x^2}\zeta,A\phi+Bu}\right\rangle,
\end{equation}
which is equal to
\begin{equation}
    -\left\langle{\frac{\partial\rho}{\partial x},A\zeta}\right\rangle+\left\langle{\frac{\partial^2\rho}{\partial x^2}\zeta,Bu}\right\rangle=
    -\left\langle{A^*p,\zeta}\right\rangle+\left\langle{\frac{\partial^2\rho}{\partial x^2}Bu,\zeta}\right\rangle.
\end{equation}
Therefore, we arrive at the following equation for time evolution of the momentum:
\begin{equation}\label{attractor_syst2p10}
    \dot p=-A^*p+\frac{\partial^2\rho}{\partial x^2}Bu.
\end{equation}
It is now easy to write down equations of motion of the compound vector $(x,\psi)$, where
$\psi=-\frac{\partial\rho}{\partial x}$ is the ``canonical'' momentum.
\begin{theorem}
The compound vector $(x,\psi)$ satisfies the {\em Hamiltonian} system of a ``maximum principle'' different from the
Pontryagin  principle:
\begin{equation}\label{attractor_syst22p110}
\begin{split}
    \dot x&=Ax+B\sign\left\langle{B,\psi}\right\rangle, \\ 
    \dot \psi&=-A^*\psi+\frac{\partial^2\rho}{\partial x^2} B\sign\left\langle{B,\frac{\partial\rho}{\partial x}}\right\rangle, 
\end{split}
\end{equation}
where the Hamiltonian is
\begin{equation}\label{hamilton2}
    \mathcal{H}=\langle{Ax,\psi}\rangle+|\langle{B,\psi}\rangle|-\left|\left\langle{B,\frac{\partial\rho}{\partial x}}\right\rangle\right|.
\end{equation}
\end{theorem}
We note that $\mathcal{H}=0$ on admissible trajectories, because
\begin{equation}
    |\langle{B,\psi}\rangle|=\left|\left\langle{B,\frac{\partial\rho}{\partial x}}\right\rangle\right| \mbox{ and } \left\langle{Ax,\frac{\partial\rho}{\partial x}}\right\rangle=0
\end{equation}
in view of the invariance of the function $\rho$ under the free motion.

\section{Motion under the basic control}\label{motion}
The control $u(x)$ is not everywhere uniquely defined, and is  a discontinuous function of $x$.

Nevertheless, a well known theorem of Filippov says that the Cauchy problem for the differential inclusion
\begin{equation}\label{sing_sys}
    \dot x=f(x),
    \qquad
    f(x)=Ax-B\sign\left\langle{B,\frac{\partial\rho}{\partial x}}\right\rangle
\end{equation}
is solvable for any initial condition $x(0)$, i.e., there exists a function $x(t)$ that is absolutely continuous, has a
given value at zero, and satisfies (\ref{sing_sys}) at points of differentiability \cite{Filippov}. This follows from the
basic properties of the function $f(x)$:
\begin{itemize}
    \item[A:] it grows linearly $|f(x)|\leq C(1+|x|)$,
    \item[B:] its values are convex compacts,
    \item[C:] it is semicontinuous as a multivalued map: if $y_n\in f(x_n)$ and $x_n\to x$, then $y\in f(x)$, where $y$ is any limit point of the sequence $y_n$.
\end{itemize}
However, the Filippov theorem does not guarantee the {\it uniqueness} of solution of the Cauchy problem. In particular,
this theorem does not allow to define a motion $x\mapsto\phi_t(x)$ under control $u(x)$ in the phase space, because the
very concept of motion stipulates uniqueness.

In this section, we show nonetheless that the motion under the control  can be defined uniquely. This is done in terms of
the DiPerna--Lions theory \cite{diperna}. First, a slight extension \cite{ovseev_diperna} of the DiPerna--Lions theory
allows one to define the motion under the singular Hamiltonian system (\ref{attractor_syst22p110}) rigorously.
\begin{theorem}\label{hamiltonian_motion}
    Consider a (singular) linear Cauchy problem for the following \underline{transport} equation in $\mathbb{R}^n$:
    \begin{equation}\label{cauchy}
        \frac{\partial v}{\partial t}=\sum b_i(x)\frac{\partial v}{\partial x_i},
        \quad
        u(x,0)=u(x)
    \end{equation}
    such that the extended DiPerna--Lions conditions are met:
    \begin{equation}\label{dip_le2a2}
        \div b\in L^\infty, \quad  b\in {W_{ *\,{\rm  loc}}^{1,1}}=BV_{\rm  loc}, \quad \frac{b(x)}{1+|x|}\in L^\infty+L^1,
    \end{equation}
    where $BV_{\rm  loc}={W_{*\,{\rm loc}}^{1,1}}$ is the Sobolev space of locally integrable functions such that their first derivatives are locally finite measures, and the rest of the notations is standard.
    Then there exists a measurable flow $x\mapsto x(t)=\phi_t(x)$ such that if $v(x)$ is a bounded measurable function, the function $v(x,t)=v(\phi_t(x))$ is the unique renormalized solution of the Cauchy problem (\ref{cauchy}).
\end{theorem}

Recall that DiPerna and Lions defined the renormalized solution of the Cauchy problem as a weak solution $v$ of the
problem such that for any smooth function $\beta:\mathbb{R}\to\mathbb{R}$ the function $\beta(v)$ is also a weak
solution. We note that for any Hamiltonian system the divergence is identically zero. The other conditions
(\ref{dip_le2a2}) can be also easily checked for singular Hamiltonian system (\ref{attractor_syst22p110}).

\begin{corollary}\label{ham_setup}
    The Cauchy problem for the transport equation that corresponds to Hamiltonian system (\ref{attractor_syst22p110}) and a bounded initial condition $v(x,p)$ has a unique renormalized solution $v$.
    The solution has the form $v(x,p,t)=v(\phi_t(x,p))$, where $\phi_t:\mathbb{R}^{4N}\to\mathbb{R}^{4N}$, $t\in\mathbb{R}$, is a uniquely defined measurable flow.
    Each curve $t\mapsto (x(t),p(t))=\phi_t(x,p)$ is absolutely continuous and satisfies (\ref{attractor_syst22p110}).
\end{corollary}

This corollary is general and useful, but it does not define any flow in the phase space $\mathbb{R}^{2N}$ of system
(\ref{sing_sys}) because in the extended symplectic space $\mathbb{R}^{4N}$ the phase space has measure zero.

The Cauchy problem for the transport equation related to ODE (\ref{sing_sys}) is as follows:
\begin{equation}\label{transport}
\frac{\partial v}{\partial t}=\left\langle Ax-B\sign\left\langle{B,\frac{\partial\rho}{\partial
x}(x)}\right\rangle,\frac{\partial v}{\partial x}\right\rangle,\quad v(x,0)=v(x).
\end{equation}
Our main result claims that in the phase space ${\mathbb{R}}^{2N}$ of the system (\ref{sing_sys}) we can define a
\textit{semiflow} which is continuous, uniquely defined everywhere, and it is related to the transport equation
(\ref{transport}) in a way the flow from Theorem \ref{hamiltonian_motion} is related to (\ref{cauchy}):
\begin{theorem}\label{motion2}
    There exists a  continuous semiflow  $x\mapsto x(t)=\phi_t(x),\,t\geq0$ such that if $v(x)$ is a bounded measurable function,
    the function $v(x,t)=v(\phi_t(x))$ is the unique renormalized solution of the Cauchy problem for the transport equation
 (\ref{transport}).
    Moreover, each curve $t\mapsto x(t)$ is absolutely continuous, and
    \begin{equation}\label{inclusion}\dot
        x(t)=Ax(t)-B\sign\left\langle{B,\frac{\partial\rho}{\partial x}(x(t))}\right\rangle,\quad x(0)=x,
    \end{equation}
    where the last equation is to be understood as a differential inclusion because the RHS is multivalued:
    $\sign(0)=[-1,1].$
\end{theorem}

The main advantage of Theorem \ref{motion2} is the continuity of the flow $\phi_t(x)$ with respect to $x$. This
continuity implies in particular that the flow is defined uniquely everywhere, although the control $u(x)$ is defined
uniquely only outside the hypersurface $\{\langle{B,{\partial\rho}/{\partial x}\rangle}=0\}$. A similar phenomenon was discovered
by I.A. Bogaevskii \cite{bogaev} for gradient differential equations $\dot x=-{\partial f}/{\partial x}$, where $f$
is a convex function. We will see in the proof of Theorem \ref{motion2} below that differential equation
(\ref{inclusion}) resembles the gradient differential equation. More precisely, the singular part of the RHS of
(\ref{inclusion}) has the form $-\alpha(x)\frac{\partial f}{\partial x},$ where $\alpha$ is a smooth nonnegative
symmetric matrix, while  $f$  is a (nonsmooth) convex function. Moreover, the quadratic form
$\langle\alpha(x)\xi,\xi\rangle+\langle x,\xi\rangle^2$ is strictly positive, and the singular part of (\ref{inclusion})
is invariant under scaling $x\mapsto\lambda x$ of the phase space. Under these circumstances it is possible to deduce
differential inequalities for 
\begin{equation}
\left\langle\alpha(x)\frac{\partial v}{\partial x},\frac{\partial v}{\partial x}\right\rangle \mbox{ and } \left\langle x,\frac{\partial v}{\partial x}\right\rangle^2, 
\end{equation}
where $v$ is a solution of (\ref{transport}), that are sufficiently
powerful to establish an a priori bound for the Lipschitz constant of $v$ in any domain of the form $\{(t,x)\in{\mathbb
R}^{2N+1}:|\phi_t(x)|\geq c\}$.

\begin{proof}
We confine ourselves to a proof of  existence of a continuous bounded solution of the transport equation
(\ref{transport}), which is obtained as a limit of classical solutions of regularized equations. Other statements can be
proved using standard techniques from \cite{diperna,ovseev_diperna}. The proof is quite long, so for the sake of clarity
we divide it into a sequence of steps.

{\bf I. Approximation by a smooth problem.} We use two approximation scales: one is controlled by parameter $n\to\infty$
such that the smooth convex function $m_n:\mathbb{R}\to\mathbb{R}$ is a uniform approximation of the function
$x\mapsto|x|$. Then, the derivative $s_n=m_n'$ approximates the $\sign$-function in $L_1$. Note that $xs_n(x)\geq0$ for
any $x\in\mathbb{R}$. Another scale is controlled by the parameter $\delta\downarrow0$, and a particular choice of the
value of $\delta$ means that we freeze the motion under system (\ref{sing_sys})  within the $\delta$-neighborhood
$U_\delta=\{\rho(x)\leq\delta\}$ of zero wrt distance $\rho$. In other words, we approximate ODE (\ref{sing_sys}) by the
nonsingular equation
\begin{equation}\label{sing_sys2}
    \dot x=Ax- Bs_n\left(\left\langle B,\frac{\partial\rho}{\partial x}\right\rangle\right)
\end{equation}
in the domain $V_\delta=\{x\in\mathbb{R}^{2N}:\rho(x)\geq\delta\}$. It is important that all the neighborhoods $U_\delta$
are invariant under the phase flow of (\ref{sing_sys2}) for positive times because the radius-function $\rho$ is
nonincreasing along the phase trajectories. Indeed, because of the analogue of equation (\ref{T}):
\begin{equation}\label{T2}
    \dot \rho=-s_n\left(\left\langle\frac{\partial {\rho}}{\partial x},B\right\rangle\right)\left\langle\frac{\partial {\rho}}{\partial x},B\right\rangle\leq0.
\end{equation}

{\bf II. Gradient form.} We rewrite equation (\ref{sing_sys2}) in the gradient form by using identity (\ref{rho02}). 
It implies that
\begin{equation}\label{identity}
    Bs_n\left(\left\langle{B,\frac{\partial\rho}{\partial x}}\right\rangle\right)=\rho\alpha(x)\frac{\partial}{\partial x}m_n\left(\left\langle{B,\frac{\partial\rho}{\partial x}}\right\rangle\right)+
    xs_n\left(\left\langle{B,\frac{\partial\rho}{\partial x}}\right\rangle\right)\left\langle{B,\frac{\partial\rho}{\partial x}}\right\rangle,
\end{equation}
which can be regarded as an approximation to
\begin{equation}\label{identity2}
    B\sign\left\langle{B,\frac{\partial\rho}{\partial x}}\right\rangle=\rho\alpha(x)\frac{\partial}{\partial x}\left|\left\langle{B,\frac{\partial\rho}{\partial x}}\right\rangle\right|+
    x\left|\left\langle{B,\frac{\partial\rho}{\partial x}}\right\rangle\right|,
\end{equation}
where $\alpha(x)=\frac{\partial^2 H}{\partial p^2}$ and $H=H_\Omega$. In particular, the ODE takes the following form:
\begin{equation}\label{ODE}
    \dot x=F(x)=f(x)-g(x)\frac{\partial}{\partial x}m_n\left(h(x)\right) \mbox{ if }x \mbox{ is in the complement $V_\delta$ of $U_\delta$},
\end{equation}
\begin{equation}\label{ODE-outside}
    \dot x=0 \mbox{ if }x \mbox{ is in  $U_\delta$}.
\end{equation}
Here the functions
\begin{equation}
\begin{split}
    &f(x)=Ax-xs_n\left(\left\langle{B,\frac{\partial\rho}{\partial x}}\right\rangle\right)\left\langle{B,\frac{\partial\rho}{\partial x}}\right\rangle,
    \\
    &g(x)=\rho\alpha(x),
    \quad 
    h(x)=\left\langle{B,\frac{\partial\rho}{\partial x}}\right\rangle
\end{split}
\end{equation}
are rather smooth: they are locally Lipschitz outside zero. equations (\ref{ODE})--(\ref{ODE-outside}) form an
approximation to (\ref{sing_sys}) rewritten in the form
\begin{equation}
    \dot x=F(x)=f(x)-g(x)\frac{\partial}{\partial x}|h(x)|,
\end{equation}
where 
$$
	f(x)=Ax-x\left|\left\langle{B,\frac{\partial\rho}{\partial x}}\right\rangle\right|, 
$$
while $g(x)$ and $h(x)$ are the same as above.

{\bf III. Derived equations.} It is important that the matrix $g=\rho\alpha$ is symmetric and nonnegative. Below we omit
the subscript $n$. The corresponding transport equation takes the form
\begin{equation}\label{transport_ODE}
    \frac{\partial v}{\partial t} =f_iv_i-g_{ij}h_jv_is(h)=F_iv_i,
\end{equation}
where $v_i=\frac{\partial}{\partial x_i}v$, $h_i=\frac{\partial}{\partial x_i}h$, $s(h)=\sign h$,
$F_i=f_i-g_{ij}h_js(h)$, and  we use Einstein's notation for summation. By differentiation, we obtain the following
equation for vector-function $V$ with components $v_k$:
\begin{equation}\label{transport_ODE_k}
    \frac{\partial v_k}{\partial t} =F_iv_{k,i}+f_{i,k}v_i -g_{ij,k}h_iv_is(h)-g_{ij}h_{jk}v_is(h)-g_{ij}h_jh_kv_i\delta(h),
\end{equation}
where $v_{k,i}=\frac{\partial v_k}{\partial x_i}$, $h_{jk}=\frac{\partial^2 h}{\partial x_j\partial x_k}$,
$g_{ij,k}=\frac{\partial g_{ij}}{\partial x_k}$, and $\delta=\delta_n$ denotes $m_n^{\prime\prime}$. equation
(\ref{transport_ODE_k}) is again a transport equation with extra terms
$f_{i,k}v_i-g_{ij,k}h_iv_is(h)-g_{ij}h_{ik}v_is(h)-g_{ij}h_ih_kv_i\delta(h)$ in the RHS. Fortunately, the most
``dangerous'' and singular term $\sigma_k=g_{ij}h_ih_kv_i\delta(h)$ has a positivity property:
\begin{equation}\label{positivity}
    g_{kl}v_l\sigma_k=g_{kl}h_kv_lg_{ij}h_jv_i\delta(h)=\left(\sum g_{kl}h_kv_l\right)^2\delta(h) \mbox{ is a positive measure.}
\end{equation}

{\bf IV. Differential inequalities.} All the other terms are linear functions of $V$ with coefficients  bounded outside
any neighborhood of zero. This implies that $w=(gV,V)=g_{kl}v_lv_k$ is a kind of quadratic Lyapunov function:
\begin{equation}\label{Lyapunov_ODE_k}
    \frac{\partial w}{\partial t} \leq F_iw_{i}+LW,
\end{equation}
where $L$ is a function uniformly bounded outside any neighborhood of zero, $W=|V|^2=\sum v_k^2$. Since the matrix
$g=\rho\alpha$ is not strictly positive definite, $W$ cannot be estimated via $w$, and equation (\ref{Lyapunov_ODE_k}) is
insufficient for establishing an a priori bound for $w$, not to mention $W$. Nonetheless, we can use the estimate
\begin{equation}\label{bound_g}
    W=\sum v_k^2\leq C\left(\left(\sum x_kv_k\right)^2+\langle{gV,V}\rangle\right),
\end{equation}
where $C$ is a positive function bounded outside any neighborhood of zero. The bound holds because the kernel of the
matrix $g(x)$ is the one-dimensional subspace of the phase space, generated by $x$. In view of equation (\ref{bound_g}),
we have to find an estimate for $z=\sum x_kv_k=Ev,$ where $E$ is the Euler operator $Ev=\sum x_k\frac{\partial
v}{\partial x_k} $. By applying the Euler operator to equation (\ref{transport_ODE_k}), we obtain:
\begin{equation}\label{transport_Euler}
    \frac{\partial z}{\partial t} =F_iEv_i+(EF_i)v_i=F_iz_i-F_iv_i+(EF_i)v_i.
\end{equation}
Here we use the commutation relation
\begin{equation}
    \frac{\partial }{\partial x_i}E=E\frac{\partial }{\partial x_i}+\frac{\partial}{\partial x_i}
\end{equation}
which implies that $Ev_i=z_i-v_i$. It is easy to compute $EF_i$: The function $F(x)=Ax- Bs\left\langle
B,\frac{\partial\rho}{\partial x}\right\rangle$ is clearly the sum of the homogeneous functions $Ax$ and
$-Bs\left\langle{B,\frac{\partial\rho}{\partial x}}\right\rangle$ of degrees 1 and 0. Therefore, $EF_i$ is a locally
bounded function. Relation (\ref{transport_Euler}) now implies that
\begin{equation}\label{transport_Euler2}
    \frac{\partial y}{\partial t}\leq F_iy_i+C'W,
\end{equation}
where $y=z^2$, and $C'$ is a locally bounded function. equation (\ref{bound_g}) says that $W\leq C\left(y+w\right)$.
Therefore, by summing inequalities (\ref{Lyapunov_ODE_k}) and (\ref{transport_Euler2}) we obtain that
\begin{equation}\label{Lyapunov_ODE_k2}
    \frac{\partial Y}{\partial t} \leq F_iY_{i}+MY,
\end{equation}
where $Y=w+y$ and the function $M$ is locally bounded outside zero uniformly wrt the scale $n$.

{\bf V. Lipschitz bounds.} equation (\ref{Lyapunov_ODE_k}) is the crucial estimate that enables us to show that the flow
$x\mapsto\Phi_t(x)=\Phi_{n,t}(x)$ corresponding to equation (\ref{ODE}) is locally Lipschitz. Importantly the
corresponding Lipschitz constant does not depend on the approximation scale $n$. Therefore, by passing to the limit
$n\to\infty$ we conclude that there exists the Lipschitz limit of $\Phi_{n,t}$, which defines the measurable semiflow
$\phi_t(x)$ of Theorem \ref{motion2} within $V_\delta$. Since $\delta$ is arbitrary, this proves in particular that the
map $x\mapsto\phi_t(x)$ is continuous if $x\neq0$ and $\phi_t(x)\neq0$.

{\bf VI. Continuity near zero.} It is in fact obvious that the map $x\mapsto\phi_t(x)$ is continuous at zero, because the
flow $\phi$ maps any neighborhood $U_\delta$ of zero into itself. It remains to consider the case $x\neq0,\,\phi_t(x)=0$.
Put $\tau=\inf\{t>0:\phi_t(x)=0\}$. It suffices to show that $\phi_\tau(y)$ is close to $\phi_\tau(x)=0$ if $y$ is
sufficiently close to $x$. We know already that for any $\epsilon>0$ the point $\phi_{\tau-\epsilon}(x)$ depends on $x$
continuously. On the other hand, it is obvious that the map $t\mapsto\phi_t(y)$ is uniformly Lipschitz for $y$ in a
neighborhood of $x$. Therefore, $|\phi_\tau(y)-\phi_\tau(x)|\leq
C|\epsilon|+|\phi_{\tau-\epsilon}(y)-\phi_{\tau-\epsilon}(x)|$. Since $\epsilon$ is arbitrary and
$|\phi_{\tau-\epsilon}(y)-\phi_{\tau-\epsilon}(x)|$ is arbitrarily small if $y$ is sufficiently close to $x$, the
continuity is proved.
\end{proof}

\noindent{\bf Remark.} One can prove the I.A.~Bogaevskii theorem \cite{bogaev} on continuous dependence of solutions to
gradient differential equations $\dot x=-\frac{\partial f}{\partial x}$ on initial conditions, where $f$ is a convex
function, in a similar but simpler way. The crucial differential inequality for the solution $v$ of the corresponding
transport equation has the form
\begin{equation}
    \frac{\partial w}{\partial t}= -\left\langle\frac{\partial w}{\partial x},\frac{\partial f}{\partial x}\right\rangle
    -2\left\langle\frac{\partial^2 f}{\partial x^2} \frac{\partial v}{\partial x},\frac{\partial v}{\partial x}\right\rangle\leq
    -\left\langle\frac{\partial w}{\partial x},\frac{\partial f}{\partial x}\right\rangle,
\end{equation}
where $w=\left|\frac{\partial v}{\partial x}\right|^2$, since $\frac{\partial^2 f}{\partial x^2}$ is a measure with
positive-definite matrix values.

\section{Asymptotic optimality of the basic control}\label{quasioptimal}

We begin with heuristic arguments. Assume that $\rho=\rho(x)$ is large, where $\rho$ is the radius-function defined in
Section \ref{property}, and that there are no resonances. By neglecting the second term in the RHS of (\ref{xi}), we get
the free motion of the vector $\phi$ governed by $\dot \phi=A\phi$. It follows from the invariance of the function $\rho$
under uncontrolled motion that the motion of $p={\partial {\rho}}/{\partial x}$ with the same accuracy is governed by the
Pontryagin equation for adjoint variables: $\dot p=-A^*p$. This follows from the Lipschitz property of the function
$\frac{\partial {\rho}}{\partial x}$, which in turn follows from the boundedness of the Hessian $\frac{\partial^2
{\rho}}{\partial x^2}$ on the ``sphere'' $\rho(x)=1$ (see Appendix \ref{rho}). The averaging amounts to finding
\begin{equation}
    \lim\limits_{\tau\to\infty}\frac1\tau\int_0^\tau \left|\left\langle p(t),B\right\rangle\right|dt.
\end{equation}
According to \cite{ovseev}, this average is the value $H_\Omega\left(p\right)$ of the support function, where $p$ is an
arbitrary point of the curve $p(t)$. By virtue of the eikonal equation (\ref{Euler}), the last expression equals 1, and
therefore, ``on average'' $\dot \rho=-1$. Using the same approximation, we obtain for any admissible control that $\dot
\rho\geq-1$ in view of (\ref{T3}). The terminating condition for the controlled motion has the form $\rho=0$. Thus,
within the framework of the assumed approximation, control (\ref{approx4}) is optimal.

\subsection{Asymptotic optimality}
A precise statement of the asymptotic optimality of control (\ref{approx4}) is as follows:
\begin{theorem}\label{main_approx}
    Suppose there are no resonances, {\it i.e.}, Eqs. (\ref{resonance}) do not hold. Consider evolution (\ref{T}) of $\rho$ under control (\ref{approx4}).
    Let
    \begin{equation}
        M=\min\{\rho(0),\rho(T),T\}.
    \end{equation}
    Then as $M\to+\infty$ we have
    \begin{equation}\label{approx_T}
        {(\rho(0)-\rho(T))}/{T}=1+o(1).
    \end{equation}
    Under any other admissible control,
    \begin{equation}\label{approx_T2}
        {(\rho(0)-\rho(T))}/{T}\leq 1+o(1).
    \end{equation}
\end{theorem}

\begin{proof}
Consider first the case where the duration of the motion $T$, although large, is much less than $\rho(T)$, meaning that
$T/\rho(T)=o(1)$. Then the controlled motion under (\ref{xi}) differs from the free one in the entire time interval
$[0,T]$ by the quantity of order $T/\rho(T)=o(1)$. Therefore, the RHS of (\ref{T3}) differs from the similar quantity for
the free motion by $o(1)$. But we have already pointed  out in the previous subsection that for the free motion, when
$p(t)=e^{-A^*t}p(0)$, the average value
\begin{equation}
    -\frac1T\int_0^T \left|\left\langle p(t),B\right\rangle\right|dt=
    -H_\Omega(p(0))+o(1)=-H_\Omega\left(\frac{\partial{\rho}}{\partial {x}}(x(0))\right)+o(1)
\end{equation}
of the RHS is $-1+o(1)$ as $T\to\infty$. Thus, the average value of the RHS of (\ref{T3}) under control (\ref{approx4})
is $-1+o(1)$ as $M\to+\infty$. By integrating the RHS, we arrive at (\ref{approx_T}). The statement (\ref{approx_T2}) can
be proved similarly.

To prove the theorem without the assumption that $T/\rho(T)$ is small, we divide the entire time interval $[0,T]$ into
many segments $[T_i,T_{i+1}]$ such that $T_{i+1}-T_i\geq M,$ and $(T_{i+1}-T_i)/\rho(T)=o(1)$, and apply to each segment
the already proved special case of the theorem. We obtain
\begin{equation}\label{approx_T22}
    \rho(T_i)-\rho(T_{i+1})=(T_{i+1}-T_i)+o(1)(T_{i+1}-T_i).
\end{equation}
Moreover, it follows from the previous arguments that the factor $o(1)$ in the last identity  is small {\em uniformly}
with respect to $i$. Summing identities (\ref{approx_T22}) on $i$ we arrive at (\ref{approx_T}). Statement
(\ref{approx_T2}) can be proved similarly.
\end{proof}

\noindent{\bf Remark.} Below we obtain a strengthening (Theorem \ref{main_approx22}) of Theorem \ref{main_approx}, where
only the initial point of the controlled motion is infinitely remote. At this point, this is impossible because if
$\rho(T)$ is not large, we can get into a standstill zone under control (\ref{approx4}). Then $\rho(T)$ does not depend
on $T$ for $T$ large, and (\ref{approx_T}) does not hold.

\subsection{Comparison with the maximum principle}

One can approach the issue of optimality of control (\ref{approx4}) from a different angle, namely by comparing the
differential equations of the motion under the control with equations (\ref{max}) of the Pontryagin maximum principle.
The following informal statement is a good guiding principle:

\medskip
    {\em The maximum principle equation for the compound vector $(x,\psi)$, where $\psi=-\frac{\partial\rho}{\partial x}$ is the
    ``canonical'' momentum, holds ``on average'' with a small error if $x$ is large.}
\medskip

Indeed, we obtain from the second equation in (\ref{attractor_syst22p110})
\begin{equation}\label{attractor_syst22p1}
    \dot\psi=-A^*\psi+\widetilde Bu, \quad \widetilde B=\frac{\partial^2\rho}{\partial x^2}B.
\end{equation}
We note that if the last equation would not contain the second term $\widetilde Bu$, then the equation for $\psi$ would
coincide with with the maximum principle equation for adjoint variables. However, the matrix
$\frac{\partial^2\rho}{\partial x^2}$ is a homogeneous function of $x$ of degree $-1$, and, according to Appendix
\ref{rho}, is bounded on the sphere $|x|=1$. Therefore, the second term has order $O\left(|x|^{-1}\right)$ for $x$ large,
and therefore, is small. We remark that the maximum condition $u=\sign\langle{B,\psi}\rangle$ holds for control
(\ref{approx4}). It remains to find out to what extent the condition $h(x,\psi)=0$ holds. We see that the motion under
control (\ref{approx4}) is governed by the Hamiltonian $\mathcal{H}$, which is very much similar to the Pontryagin
Hamiltonian $h(x,\psi)$. The difference between the Hamiltonians is $1-\left|\left\langle{B,\partial\rho/\partial
x}\right\rangle\right|$. The arguments of the previous section imply that the difference is zero ``on average'' in the
non-resonant case. Indeed, the average value of $\left|\left\langle{B,\partial\rho/\partial x}\right\rangle\right|$ is
close to 1 for $x$ sufficiently large, as it is shown in the proof of Theorem \ref{main_approx}.

\section{Efficiency of  basic control at finite distance from zero}\label{efficiency}

We already know that asymptotically the time of motion from the level set $\rho=M$ to the level set $\rho=N$ under
control (\ref{approx4}) is $(M-N)(1+o(1))$ if $M,N,$ and $M-N$ are very large. Now we show that a nonasymptotic estimate
holds: the time of motion $T$ is $O(M-N)$, if $M,N$ and $M-N$ are greater than a constant $C(\underline\omega)$ that
depends only on parameters ${\underline\omega}=(\omega_1,\dots,\omega_N)$ of our system of oscillators. equation
(\ref{T}) could be rewritten using notation of the previous section as
\begin{equation}
    \dot\rho=-|\langle p,B\rangle|,
\end{equation}
and this reduces the required estimate to the inequality
\begin{equation}\label{T33}
    \int_0^T|\langle p,B\rangle|dt\geq cT,
\end{equation}
where $c=c(\underline\omega)$ is a strictly positive constant. The proof of inequality (\ref{T33}) below is a direct
application of the perturbation theory of completely observable time-invariant linear systems (Appendix
\ref{perturb_observ}).

\begin{theorem}\label{observation5}
    Suppose that the motion from the level set $\rho=M$ to the level set $\rho=N$ under control (\ref{approx4}) proceeds within the domain $\rho(x)\geq C(\underline\omega)$,
    in the time interval of integer length $T$, where $C(\underline\omega)$ is a (sufficiently large) constant that depends only on the eigenfrequencies.
    Then $T\leq c(M-N),$ where $c=c(\underline\omega)$ is a strictly positive constant.
\end{theorem}
\begin{proof}
We regard (\ref{attractor_syst22p1}) as a definition of a completely observable linear system, where, using notation of
Theorem \ref{observation0}, the phase vector is $x=p$ and matrices are $\alpha=-A^*,\,\beta=B^*$, observation is
$y=B^*p=\langle{p,B}\rangle$, and perturbation is $f={\widetilde B}u$. Assume that in the entire time interval $I$ of
integer length $T$ the motion of the state vector $x$ takes place within the domain $\rho(x)\geq C$. Then $|f|=O(1/C)$ in
the entire interval. Moreover, the eikonal equation (\ref{Euler}) holds for $p$, and, therefore, $1\ll|p|$ and
$T\ll\int_I|p|dt$ (here $\ll$ is the Vinogradov symbol, meaning $O(\rm RHS)$). The estimate of Theorem \ref{observation0}
from Appendix \ref{perturb_observ} gives that
\begin{equation}
    T\ll\int_I|p|dt\ll \int_I|\langle p,B\rangle|dt+\frac1C T.
\end{equation}
By taking a sufficiently large constant $C=C(A,B)$, we obtain that
\begin{equation}
    T\ll\int_I|\langle p,B\rangle|dt=M-N.
\end{equation}
This inequality is the same as (\ref{T33}) up to a notational change.
\end{proof}

We emphasize that Theorem \ref{observation5} holds both in the resonant and in the non-resonant cases. Indeed, we need
not worry about the linear relation between the eigenfrequencies, only the Kalman condition $\omega_i\neq\omega_j$ is
relevant. It is easy to establish what happens when we apply the scaled control
\begin{equation}\label{controlU}
    u_U(x)=Uu(x),\,|U|\leq1.
\end{equation}

\begin{theorem}\label{observation52}
    Suppose that the motion from the level set $\rho=M$ to the level set $\rho=N$ under the control (\ref{approx4}) proceeds within the domain $\rho(x)\geq UC(\underline\omega),$ in the time interval of integer length $T$,
    where $C(\underline\omega)$ is a (sufficiently large) constant from (\ref{observation5}) that depends on the eigenfrequencies.
    Then $T\leq\frac{c}{U}(M-N)$, where $c=c(\underline\omega)$ is the constant from (\ref{observation5}).
\end{theorem}
\begin{proof}
The statement follows from the previous theorem upon the uniform scaling $x\mapsto Ux$ of the phase space.\end{proof}

\section{Singular trajectories}\label{attractor}

We know that if the system under control (\ref{approx4}) goes sufficiently far from the target, i.e., the equilibrium,
then the control is efficient, meaning that we approach the target with a positive speed. However, within a zone close to
the equilibrium, there could arise $\omega$-limit sets (attractors), so that by moving along them we do not approach the
target. It is clear that the control should be changed before getting into an attractor. In fact, the attractors define
an exact bound for the efficiency zone of the control.

\subsection{Standstill zone}\label{standstill}
The simplest attractor is a singleton, i.e., a fixed point. We call the set of such points the standstill zone. There is
an obvious upper bound for standstill zones for any admissible control bounded by a constant $U$, namely, this is the
interval
\begin{equation}\label{bound}
\begin{array}{l}
    \{A^{-1}Bu,\,|u|\leq U\}= \{y_{i}=0,\, \omega_i^{2}x_{i}=\omega_j^{2}x_{j},\,
    |\omega_i^{2}x_{i}|\leq U,\, \,i,j=1,\dots,N\}. \\
\end{array}
\end{equation}

\subsection{Motion along an attractor}\label{attractor11}
More generally, consider the motion under control (\ref{approx4}) along an attractor. It follows immediately from
(\ref{T}) and (\ref{xi}) that it is governed by the system
\begin{equation}\label{attractor_syst}
    \dot\rho=0, \quad  \dot\phi=A\phi+\frac{1}{\rho}Bu,
\end{equation}
and the constraint $\langle\frac{\partial\rho}{\partial x},B\rangle=0$. Taking the relation
$\langle\frac{\partial\rho}{\partial x},A\phi\rangle=0$ from the beginning of Section (\ref{quasioptimal}) into account,
we immediately derive the following expression for the control:
\begin{equation}\label{attractor_u}
    u=u(\phi)=-\rho\frac{\left\langle{\frac{\partial^2\rho}{\partial x^2}A\phi,B}\right\rangle}{\left\langle{\frac{\partial^2\rho}{\partial x^2}B,B}\right\rangle},
\end{equation}
where $\frac{\partial^2\rho}{\partial x^2}$ is the Hessian of the function $\rho$.

We conclude that the motion along an attractor is governed by equation
\begin{equation}\label{attractor_syst1}
    \dot \rho=0,\quad  \dot \phi=A\phi+Bf(\phi).
\end{equation}
More precisely, an integral curve of system (\ref{attractor_syst1}) is contained in the attractor, if the inequality
$|f|\leq1/\rho$ holds along the curve. Note that the nontrivial existence and uniqueness issues for the integral curve is
already resolved by Theorem \ref{motion2}.

Thus, we get the following description of singular arcs of control (\ref{approx4}). Consider the dynamical system on the
manifold
\begin{equation}\label{attractor_syst11}
    \sigma=\left\{\rho=1,\left\langle\frac{\partial\rho}{\partial x},B\right\rangle=0\right\}
\end{equation}
of dimension $2N-2$, described by the equation
\begin{equation}\label{attractor_syst2}
    \dot \phi=A\phi+Bf(\phi).
\end{equation}
Then, if the inequality $|f|\leq1/\rho$ holds along an  $\omega$-limit set $\mathfrak O$ of the system, the set
$\rho\mathfrak O$ is an attractor for the motion under (\ref{approx4}). Conversely, any attractor of the controlled
motion can be obtained in the same way from   dynamical system (\ref{attractor_syst2}). In particular, we obtain a
criterion for absence of nontrivial attractors in the form of the inequality for  ``radius'', given by the following
theorem.
\begin{theorem}\label{attractor_free}
    The domain $\rho\geq\mu^{-1}$, where $\mu$ is the minimum over all attractors of system (\ref{attractor_syst2}) of the maximum of $|f|$ over the attractor, is attractor-free,
    i.e., does not contain nontrivial minimal $\omega$-limit sets of system (\ref{T})--(\ref{xi}).
\end{theorem}

The value of the minimax  $\mu$ is a primary characteristic of system (\ref{attractor_syst2}). Its importance is due to
the fact that it gives an exact bound for the efficiency zone for control (\ref{approx4}).

The next theorem follows in a formal way from Theorem \ref{observation5}.
\begin{theorem}\label{mu}
    Suppose that $\epsilon>0$ and the motion under control (\ref{approx4}) in a sufficiently long time interval $[a,b]$ of length $T$ proceeds within the domain $\{\rho\geq\mu^{-1}+\epsilon\}$.
    Then, $\rho(a)-\rho(b)\geq c(\epsilon)T,$ where $c(\epsilon)$ is a positive constant.
    On the other hand, there are infinitely long motions within $\{\mu^{-1}-\epsilon\leq\rho\leq\mu^{-1}\}$, where $\rho(t)$ is a constant.
\end{theorem}

In the notation of Theorem \ref{observation5}, this means that $C(\underline\omega)= \mu^{-1}+\epsilon,$ and stresses the
importance of finding a lower estimate for $\mu$.

We note that the manifold $\sigma$ is diffeomorphic to a $(2N-2)$-dimensional sphere. In particular, for the case of two
oscillators the problem of the value of  $\mu$ reduces to the classical problem of examination of a dynamical system on
the two-dimensional sphere.

It is convenient to study the dynamical system ``dual'' to (\ref{attractor_syst2}), which describes the motion of vector
$p=\frac{\partial\rho}{\partial x}(\phi)$. By defining  $\widetilde B=\frac{\partial^2\rho}{\partial x^2}B$, we obtain
from (\ref{attractor_syst22p1}) the following system
\begin{equation}\label{attractor_syst22p}
    \dot p=-A^*p+\widetilde Bu, \quad \widetilde B=\frac{\partial^2\rho}{\partial x^2}B.
\end{equation}
The matrix  $\frac{\partial^2\rho}{\partial x^2}$ in the equation can be rewritten as a function of $p$. To do this, we
use  relation (\ref{rho03}) between the second derivatives of the dual functions $H$ and $\rho$. In particular,  taking
identities $\left\langle{{\partial\rho}/{\partial x},B}\right\rangle=0$ and $\rho=1$ into account, we obtain for
$\zeta=B$ that
\begin{equation}\label{rho4}
    \widetilde B=\left(\frac{\partial^2 {H}}{\partial p^2}\right)^{-1}B.
\end{equation}
Moreover, the condition
\begin{equation}
    \langle{p,B}\rangle=\left\langle{\left(\frac{\partial^2\rho}{\partial x^2}\right)^{-1}p,\widetilde B}\right\rangle=0
\end{equation}
is fulfilled in the motion along attractor, which means that
\begin{equation}
    u={\langle{p,AB\rangle}}{\left\langle{\frac{\partial^2\rho}{\partial x^2}B,B}\right\rangle}^{-1}.
\end{equation}
Note that, provided that $\rho(x)=1$, the value of $b={\left\langle{\frac{\partial^2\rho}{\partial
x^2}(x)B,B}\right\rangle}$ has a uniform upper estimate:
\begin{equation}
    b\leq C(A)|B|^2,
\end{equation}
where $C(A)$ is a positive constant that depends only on the matrix $A$ of the system considered. Therefore, in order to
estimate $\mu$ from below it suffices to estimate from below the minimum  $\widetilde \mu$ over all attractors of system
(\ref{attractor_syst22p}) of the maximum of the function ${\widetilde f}(p)=|\langle{p,AB}\rangle|$ on the attractor.

\subsection{Bound for the attractor-free domain}\label{attractor2}
According to Theorems \ref{attractor_free} and \ref{mu}, any lower bound for the constant $\mu$ gives a lower bound for
the attractor-free domain.
\begin{theorem}\label{73}
    Let $\mu$ be the minimum over trajectories of (\ref{attractor_syst2}) of the maximum of the function $|f|$ on a trajectory.
    Then the number $\mu$ is strictly positive.
\end{theorem}

\begin{proof}
According to Theorem \ref{attractor_free}, we need to find a lower bound for the constants $\mu$ or $\widetilde \mu$. One
can easily approach the problem using perturbation theory of observable systems (Theorem \ref{observation0}). Indeed,
suppose that the maximum of the function ${\widetilde f}(p)=|\langle p,AB\rangle|$ on an attractor is less than $c$. Then
in particular, the vector $p$, solution of system (\ref{attractor_syst22p}), satisfies the equation $\dot p=-A^*p+f,$
where $|f|\ll c$ in a time interval of arbitrary length. Consider the observable coordinate $\langle{p,B}\rangle$ which
is identically zero on the manifold
\begin{equation}\label{sigma2}
    \sigma\mbox{\v{}}=\left\{p\in\mathbb{R}^{2N}: H(p)=1, \langle p,B\rangle=0\right\}
\end{equation}
where the motion takes place. The a priori bound of Theorem \ref{observation11}, applied to an interval of unit length,
shows that
\begin{equation}
    1\ll\int|p|dt\ll c,
\end{equation}
and gives the required bound for $c$.
\end{proof}

\section{Feedback near the terminal point}\label{terminal}
\subsection{Asymptotic theory of reachable sets as $T\to0$}
The  design of the basic control in the high-energy zone is based on the asymptotic behavior of reachable sets
$\mathcal{D}(T)$ as $T\to\infty$. We take a natural approach to feedback control design near the equilibrium point, by
considering the asymptotic behavior of the reachable set $\mathcal{D}(T)$ of system (\ref{syst1})--(\ref{syst2}) as
$T\to0$. This problem was studied in detail for linear systems in \cite{ovseev_gonch}. The conclusion of this
investigation is that the general picture of the asymptotic behavior of the reachable set $\mathcal{D}(T)$ is the same
for all linear systems, so it suffices to study only a single canonical system.

Recall that the  Banach-Mazur distance $d$ between two zero-centered convex bodies $\Omega _1,\Omega _2$ in a vector
space $V$ is defined as
\begin{eqnarray}\label{B-M}
   \quad\quad d(\Omega _1,\Omega _2)=\log (t(\Omega _1,\Omega _2)t(\Omega _2,\Omega _1)), \quad t(\Omega_1 ,\Omega_2 ) = \inf \{t \geq 1: t\Omega_1\supset \Omega_2 \}.
\end{eqnarray}

The main result of \cite{ovseev_gonch} can be restated as follows:
\begin{theorem}\label{thmain} Suppose that  system (\ref{syst1}) in space $V$ is controllable.
    Then there are matrices $\Delta(T)$ and a fixed convex body $\Omega\subset V$ such that the asymptotic equivalence ${\mathcal D}(T)\sim \Delta(T)\Omega$ holds.
    Moreover, $d({\mathcal D} (T),\Delta(T)\Omega)=O(T)$.
\end{theorem}
This equivalence means that the Banach-Mazur distance between  the RHS and the LHS of the asymptotic equality tends to 0
as $T\to0$.

The idea of our approach is to design a control by using, instead of the reachable sets ${\mathcal D} (T)$, a family of
ellipsoids ${\mathcal E} (T)$ with a similar basic property ${\mathcal E} (T)= \Delta(T){\mathcal E}$, where ${\mathcal
E}$ is a fixed (time-invariant) ellipsoid. It turns out that the quadratic function that defines the crucial ellipsoid
${\mathcal E}$ is a common Lyapunov function for two explicitly constructed linear systems.

\subsection{Common Lyapunov functions}
The design of our local feedback control goes back to \cite{korobov}. It uses a preliminary reduction of system
(\ref{syst1})--(\ref{syst2}) to a canonical form by means of transformations
\begin{equation}\label{transformations}
    A\mapsto A+BC,\quad  u\mapsto u-Cx ,\quad A\mapsto D^{-1}AD,\quad B\mapsto D^{-1}B,
\end{equation}
that correspond to adding a linear feedback control, and to coordinate changes (gauge transformations). We state the
result as follows:
\begin{lemma}\label{canonical}
    System (\ref{syst1})--(\ref{syst2}) can be reduced by transformations (\ref{transformations}) to the following canonical form:
\begin{equation}\label{AB0}
    \dot {\mathfrak x}={\mathfrak A}{\mathfrak x}+{\mathfrak B}{\mathfrak u},
\end{equation}
\begin{equation}\label{AB}
\begin{array}{c}
 \mathfrak{A} = \left( {\begin{array}{ccccc}
    0 & &  &   \\
    -1 & 0 &  &   \\
    & -2 &0&\\
    & &\ddots & \ddots  \\
    &&  & -2N+1 & 0  \\
\end{array} } \right), \quad
\mathfrak{B} = \left( \begin{gathered}
    1 \hfill \\
    0 \hfill \\
    0 \hfill \\
    \vdots \hfill \\
    0 \hfill \\
\end{gathered}  \right).
 \\
\end{array}
\end{equation}
The matrix of the linear feedback should be chosen in the form
\begin{equation}\label{C}
   C=(c_1\,0\,c_2\,0\,\dots\, c_N\,0),\quad 
   c_k=(-1)^{N+1}\omega_k^{2N}\prod_{i\neq k}(\omega_i^{2}-\omega_k^{2})^{-1}.
\end{equation}
The gauge matrix $D$ transforms the standard basis $e_i=(\delta_{ij})$ of $\mathbb{R}^{2N}$ into the basis
\begin{equation}\label{e}
    \mathfrak{e}_i=\frac{(-1)^{i-1}}{(i-1)!}(A+BC)^{i-1}B,\,i=1,\dots,2N,
\end{equation}
and has the following form. Define $2\times2$ matrices
\begin{equation}\label{d}
    d_{ij}=(-1)^{j-1}\lambda_i^{j-1}\left(%
    \begin{array}{cc}
     0 &  -\frac{1}{(2j-1)!} \\
    \frac{1}{(2(j-1))!} & 0
    \end{array}\right), \mbox{ where }
    \lambda_k=\sum_{i\neq k}\omega_i^2.
\end{equation}
Then,
\begin{equation}\label{D}
    D\mbox{ is the $N\times N$ matrix $(d_{ij})$ of $2\times2$ blocks $d_{ij}.$}
\end{equation}
\end{lemma}

When regarded as an existence theorem of a canonical form, without explicit formulas for matrices $C$ and $D$, Lemma
\ref{canonical} is a particular case of the Brunovsky theorem \cite{brun}. We give a proof of the lemma in Appendix
\ref{brunovsky_lemma}.

Following \cite{anan}, introduce a matrix function of time related to system (\ref{AB}):
\begin{equation}\label{delta}
    \delta(\mathfrak T)=\diag({\mathfrak T}^1,{\mathfrak T}^{2}, \dots, {\mathfrak T}^{2N})^{-1}.
\end{equation}
Below the parameter ${\mathfrak T}$ will be a function ${\mathfrak T}={\mathfrak T}(\mathfrak{x})$ of the phase vector.
Define the matrices in accordance with \cite{anan,korob}
\begin{equation}\label{M}
\begin{array}{l}
    \mathfrak{q}=(\mathfrak{q}_{ij}),\, \mathfrak{q}_{ij}=\int_0^1 x^{i+j-2}(1-x)dx=[(i+j)(i+j-1)]^{-1}, \\[1em]
    \mathfrak{Q}=\mathfrak{q}^{-1},\quad
    \mathfrak{C}=-\frac{1}{2}\mathfrak{B}^{*}\mathfrak{Q}, \quad \mathfrak{M}=\diag(1, 2, \dots, 2N). \\
\end{array}
\end{equation}
Define the feedback control by the equation
\begin{equation}\label{u}
    {\mathfrak u}(\mathfrak{x})=\mathfrak{C}\delta({\mathfrak T}( \mathfrak{x}))\mathfrak{x},
\end{equation}
where the function ${\mathfrak T}={\mathfrak T}(\mathfrak{x})$ is defined implicitly by the following relation:
\begin{equation}\label{condu}
    \left\langle{\mathfrak{Q}\delta({\mathfrak T})\mathfrak{x},\delta({\mathfrak T})\mathfrak{x}}\right\rangle=\kappa^2.
\end{equation}
The value of the positive constant $\kappa$ will be chosen below. A basic result on the steering of the canonical system
(\ref{AB0})--(\ref{AB}) to zero is as follows:
\begin{theorem}\label{main}
    The following statements hold true:
    \begin{itemize}
        \item[A:] The matrix  $\mathfrak{Q}$ defines a common quadratic Lyapunov function for the matrices $-\mathfrak{M}$ and $\mathfrak {A+BC}.$
        \item[B:] equation (\ref{condu}) defines ${\mathfrak T}={\mathfrak T}(\mathfrak{x})$ uniquely.
        \item[C:] Control (\ref{u}) is bounded: $|{\mathfrak u}|\leq\frac{\kappa}{2}\sqrt{\mathfrak{Q}_{11}} $.
        \item[D:] Control (\ref{u}) brings the point $\mathfrak{x}$ to $0$ in time ${\mathfrak T}(\mathfrak{x})$.
    \end{itemize}
\end{theorem}

\begin{proof}
{\bf Statement A} amounts to the matrix inequalities
\begin{equation}\label{lmi}
 \{\mathfrak{M,q}\}>0,\,
 \{\mathfrak{A,q}\}-\frac12\{\mathfrak{B},\mathfrak{B}^*\}<0,
\end{equation}
where we use the  ``Jordan brackets'' $\{\alpha,\beta\}=\alpha\beta+\beta^*\alpha^*$. Indeed, if
$Q(x,x)=\langle{Qx,x}\rangle$ is a quadratic Lyapunov function for a stable matrix $A,$ this implies the matrix
inequality $\{Q,A^*\}<0$, or, in other words, the relation
\begin{equation}
    \{A,Q^{-1}\}=Q^{-1}\{Q,A^*\}Q^{-1}<0.
\end{equation}
Moreover, the matrix $\frac12\{A,Q^{-1}\}$ corresponds to the negative quadratic form
\begin{equation}
    Q^{-1}(x,A^*x)
\end{equation}
A straightforward computation shows that $\{\mathfrak{BC,q}\}=-\frac12\{\mathfrak{B},\mathfrak{B}^*\}$.

We implement the phase space  $\mathbb{R}^{2N}$ as the space of polynomials $f$ of degree less than $2N$ in the variable
$x$. Then the canonical basis $\mathfrak{e}_k$ of $\mathbb{R}^{2N}$ is represented by the monomials $m_k(x)=x^{k-1}$.
Note that the matrix $\mathfrak{A}^*$ is represented by the differentiation operator $f\mapsto-\frac{\partial}{\partial
x}f$, while the matrix $\mathfrak{M}^*=\mathfrak{M}$ is represented by the operator $f\mapsto\frac{\partial}{\partial
x}xf$. The dual vector $B^*=(1,0,\dots,0)$ is represented by the functional $f\mapsto f(0)$. Consider relations
(\ref{lmi}) in the functional model. The quadratic form $\mathfrak q(f,f)$, related to the matrix $\mathfrak q$, takes
the form $\int_0^1 f^2(x)(1-x)dx$. It is a positive form. The matrices $\{\mathfrak M,\mathfrak q\},\,\{\mathfrak
A,\mathfrak q\},\,\{\mathfrak{B},\mathfrak{B}^*\}$ are represented by the following quadratic forms in the functional
model:
\begin{equation}\label{forms}
    \begin{array}{l}
        \mu(f)=\mathfrak q(f,\mathfrak M^*f)=2\int \left(\frac{\partial}{\partial x}xf\right)(x)f(x)(1-x)dx,\\[1em]
        \alpha(f)=\mathfrak q(f,\mathfrak A^*f)=-2\int \left(\frac{\partial}{\partial x}f(x)\right)f(x)(1-x)dx,\,\beta(f)=2f(0)^2,
    \end{array}
\end{equation}
where the integration is over the interval $[0,1]$. Integrating by parts, we obtain
\begin{equation}\label{alpha}
\begin{array}{lll}
    &\alpha(f)&=-\int \frac{\partial}{\partial x}f^2(x)(1-x)dx=-\int f^2(x)dx+f^2(0)\\[1em]
    &\mu(f)&=2\int f^2(x)(1-x)dx-\int f^2(x)[(1-x)x]'dx=\\[1em]&&=2\int f^2(x)[(1-x)+\frac12(2x-1)]dx=\int f^2(x)dx.
\end{array}
\end{equation}
Therefore, $\alpha(f)-\frac12\beta(f)=-\mu(f)$, and both sides of the latter equality coincide with the negative
quadratic form $-\int f^2(x)dx$. This proves inequalities (\ref{lmi}), and Statement A of Theorem. Moreover, we have
shown that
\begin{equation}\label{forms2}
    -\{\mathfrak{M,q}\}=\{\mathfrak{A,q}\}+\{\mathfrak{BC},\mathfrak{q}\}.
\end{equation}
The last relation is equivalent to the equality of quadratic forms
\begin{equation}\label{forms3}
    \left\langle{\mathfrak{Q}y,\mathfrak{[A+BC]}y}\right\rangle=-\left\langle{\mathfrak{Q}y,\mathfrak{M}y}\right\rangle
\end{equation}

We note that the proceeding arguments can be easily generalized to the case when the matrix $\mathfrak q$ is represented
by a quadratic form
\begin{equation}
    \int_0^{\infty} f^2(x)q(x)dx,
\end{equation}
where the nonnegative function $q$ is monotone nonincreasing ($q'\leq0$), decreases at infinity faster than any power of
$x$, and satisfies $q(0)=1$. Indeed, the matrices
\begin{equation}
    \{\mathfrak M,\mathfrak q\},\quad \{\mathfrak A,\mathfrak q\},\quad \{\mathfrak{B},\mathfrak{B}^*\}
\end{equation}
correspond in the functional model to the following quadratic forms:
\begin{equation}\label{forms22}
    \begin{array}{l}
        \mu(f)=\mathfrak q(f,\mathfrak M^*f)=2\int \left(\frac{\partial}{\partial x}xf\right)fqdx,\\[1em]
        \alpha(f)=\mathfrak q(f,\mathfrak A^*f)=-2\int \left(\frac{\partial}{\partial x}f\right)fqdx,\,\beta(f)=2f(0)^2,
    \end{array}
\end{equation}
where the integration is over the ray  $[0,+\infty)$. Integrating by parts, we obtain
\begin{equation}\label{alpha22}
    \begin{array}{lll}
        &\alpha(f)&=-\int \left(\frac{\partial}{\partial x}f^2\right)qdx=\int f^2q'dx+f^2(0)q(0)\\[1em]
        &\mu(f)&=-2\int xf(f'q+fq')dx=-\int \left(\left(\frac{\partial}{\partial x}{f^2}\right)xq+2f^2xq'\right)dx\\[1em]&&=\int f^2(q-xq')dx.
    \end{array}
\end{equation}
Thus, inequalities (\ref{lmi}) hold true.

\noindent {\bf Statement B} follows from strict monotonicity of the function ${\mathfrak T}\mapsto
\left\langle{\mathfrak{Q}\delta({\mathfrak T})\mathfrak{x},\delta({\mathfrak T})\mathfrak{x}}\right\rangle$ which in turn
follows immediately from the first inequality in (\ref{lmi}).

\noindent {\bf Statement C} follows from the Cauchy inequality. Indeed,
$\mathfrak{u}=-\frac12\left\langle{\mathfrak{Q}\mathfrak{B},y}\right\rangle$, where $y=\delta({\mathfrak T})\mathfrak{x}$
and $\left\langle{\mathfrak{Q}y,y}\right\rangle=\kappa^2$. Therefore
\begin{equation}
    |\mathfrak{u}|\leq\frac12\left\langle{\mathfrak{Q}y,y}\right\rangle^{1/2}\left\langle{\mathfrak{Q}\mathfrak{B},\mathfrak{B}}\right\rangle^{1/2}
    \leq\frac{\kappa}{2}\left\langle{\mathfrak{QB},\mathfrak{B}}\right\rangle^{1/2}
    =\frac{\kappa\sqrt{\mathfrak{Q}_{11}}}{2}.
\end{equation}

\noindent {\bf Statement D} follows by computing of the total derivative $\dot {\mathfrak T}$. Letting
$\delta=\delta({\mathfrak T})$, we obtain
\begin{equation}\label{prop}
    \delta \mathfrak{A}\delta^{-1}={\mathfrak T}^{-1}A,\,\, \delta \mathfrak{B}={\mathfrak T}^{-1}\mathfrak{B},\,\,\frac{d}{d{\mathfrak T}}\delta=-{\mathfrak T}^{-1}\mathfrak{M}\delta,
\end{equation}
which immediately implies for $y=\delta({\mathfrak T})\mathfrak{x}$, the equation
\begin{equation}\label{y2}
    \dot y={\mathfrak T}^{-1}\left(\mathfrak{A}y+\mathfrak{B}u-\dot {\mathfrak T}\mathfrak{M}y\right).
\end{equation}
Then it follows from relations (\ref{u}) and (\ref{condu}) that
\begin{equation}\label{T222}
    \left\langle{\mathfrak{Q}y,\mathfrak{[A+BC]}y-\dot {\mathfrak T}\mathfrak{M}y}\right\rangle=0,
\end{equation}
but in view of (\ref{forms3}), this implies  $\dot{\mathfrak T}=-1.$

We note that in a more general situation where the matrix $\mathfrak q$ is related to a quadratic form $\int_0^{\infty}
f^2(x)q(x)dx$, Statement D is valid iff $q'=-(q-xq')$. This implies easily that $q=(1-x)_+,$ so Statement D characterizes
the matrix $\mathfrak q$ of this kind essentially uniquely.
\end{proof}

{\bf Remark.} Suppose that $\tau(\mathfrak{x})$ is the minimum time for steering a state $\mathfrak{x}$ of the canonical
system (\ref{AB0}) to zero by using any admissible control $v,\,|v|\leq1$. Then ${\mathfrak T}(\mathfrak{x})$ and
$\tau(\mathfrak{x})$ are comparable, meaning that $1\leq{{\mathfrak T}(\mathfrak{x})}/{\tau(\mathfrak{x})}\leq C,$ where
$C$ is a constant. This follows from equation (\ref{condu}) and the fact, that the matrix $\delta({\mathfrak T})^{-1}$
brings the reachable set $\mathcal{D}(1)$ of the canonical system (\ref{AB0}) to $\mathcal{D}({\mathfrak T})$:
$\delta({\mathfrak T})\mathcal{D}({\mathfrak T})=\mathcal{D}(1)$ in the unit time.

Theorem \ref{main} was obtained in \cite{korob} in a less precise form. Our proof is about ten times shorter. Moreover,
the method applied allows us to indicate a large class of common quadratic Lyapunov functions for the matrices
$-\mathfrak{M}$ and $\mathfrak {A+BC}$. The two number-theoretic results below are not directly related to control
problems.
\begin{theorem}\label{main2}
    The matrix $\mathfrak Q$ is even integer: $\mathfrak Q\in 2M_{2N}(\mathbb{Z}).$
\end{theorem}

A strengthening of the above result is related to the value of the matrix element $\mathfrak{Q}_{11}$:
\begin{theorem}\label{main3}
    The matrix element $\mathfrak{Q}_{11}=2N(2N+1)$.
\end{theorem}

We prove these theorems in Appendices \ref{proof_2}--\ref{proof_3}. Both proofs are based on the consideration of
orthogonal polynomials. This idea goes back at least to Hilbert \cite{hilbert}.

\begin{corollary}\label{main31}
Control (\ref{u}) is bounded by $\frac{\kappa}{2}\sqrt{2N(2N+1)}$.
\end{corollary}

The corollary is obvious. Numerical experiments suggest the following:

{\it $\mathfrak{Q}_{11}$ is a divisor of all the elements of the matrix $\mathfrak Q$:  $\mathfrak
Q\in\mathfrak{Q}_{11}M_{2N}(\mathbb{Z})$}.

The explicit form of $\mathfrak{Q}=\mathfrak{q}^{-1}$ in the $4$-dimensional case is
\begin{equation}\label{Q2}
20\times\left( {\begin{array}{*{20}c}
   {1}       & { -9}  & {21}      & {-14}   \\
   { - 9} & {111}    & {-294}  & {210}  \\
   {21}     & { -294} & {840}   & {-630} \\
   { -14}  & {210}   & {-630} & {490}   \\
\end{array}} \right)
\end{equation}
(we note that $\mathfrak Q_{11}=20$). This gives equation (\ref{condu}) in the following explicit form:
\begin{equation}
\begin{array}{l}
{\mathfrak T}^8-20{\mathfrak x}_{1}^2{\mathfrak T}^6+360{\mathfrak x}_{1}{\mathfrak x}_{2}{\mathfrak T}^5-(2220{\mathfrak
x}_2^2+840{\mathfrak x}_1{\mathfrak x}_3){\mathfrak T}^4+(11760{\mathfrak x}_{2}{\mathfrak x}_{3}+ \\[1em]
560{\mathfrak x}_{1}{\mathfrak x}_{4}){\mathfrak T}^3 - (8400{\mathfrak x}_{2}{\mathfrak x}_{4}+16800{\mathfrak
x}^{2}_{3}){\mathfrak T}^2+25200{\mathfrak x}_{3}{\mathfrak x}_{4}{\mathfrak T}-9800{\mathfrak x}_{4}^2=0.
\end{array}
\end{equation}
A tight bound for the absolute value of control (\ref{u}) is $\frac\kappa2\sqrt {\mathfrak{Q}_{11}}=\kappa\sqrt5 $, where
$\kappa$ is the constant from (\ref{condu}). If we want that  $|{\mathfrak u}|\leq1/2$, we put $\kappa=(2\sqrt5)^{-1}$.
This is the bound we use at the terminal stage of the control.

\section{Control matching}\label{match}
In Section \ref{terminal}, we designed a local feedback control that works in the neighborhood of zero. The switching to
this control should occur at the boundary of an \textit{invariant} domain with respect to the phase flow so that the
local feedback control can be applied within the interior. We confine ourselves to the invariant domains of the form
\begin{equation}\label{theta}
    G_\Theta=\{{\mathfrak x}: {\mathfrak T}({\mathfrak x})\leq\Theta\}=\{{\mathfrak x}: \langle\mathfrak{Q}\delta(\Theta){\mathfrak x},\delta(\Theta){\mathfrak x}\rangle\leq1\}.
\end{equation}
The invariant domain  $G_\Theta$ should satisfy two conditions:
\begin{itemize}\label{conditionsG}
    \item[A:] The domain $G_\Theta$ contain the inefficiency domain $\{\rho(x)\leq UC(\underline\omega)\}$ of the preceding control;
    \item[B:] The domain  $G_\Theta$ is contained in the strip $\{|Cx|\leq1/2\}$, where $C$ is the matrix (\ref{C}).
\end{itemize}
Condition B allows one to use at the terminal stage controls ${\mathfrak u}$ which are less than 1/2 in absolute value.
Therefore, the constant $\kappa^2$ in (\ref{condu}) should be equal to $(2N(2N+1))^{-1}$. If we applied at the preceding
stage the control (\ref{controlU}), Condition A says that the set $UC(\underline\omega)\Omega$ is contained in
$G_\Theta$. Here $C(\underline\omega)$ is the estimate for the ``radius'' of the attractor-free domain found in
Subsection \ref{attractor2}. In other words, the following inequality should be fulfilled for the support functions:
\begin{equation}\label{condA}
    UC(\underline\omega)H_\Omega(D^*p)\leq\left\langle{\delta(\Theta)^{-1}\mathfrak{q}\delta(\Theta)^{-1}p,p}\right\rangle^{1/2},
\end{equation}
where $D$ is the matrix (\ref{D}). It is clear that the inequality holds, provided that  $U$ is sufficiently small.

Condition B says precisely that the value of the support function of the ellipsoid $G_\Theta$ at the vector ${D^*}^{-1}C$
does not exceed $1/2$ in absolute value. In other words,
\begin{equation}\label{condB}
    \left\langle\delta(\Theta)^{-1}\mathfrak{q}\delta(\Theta)^{-1}{D^*}^{-1}C,{D^*}^{-1}C\right\rangle^{1/2}\leq1/2.
\end{equation}
Certainly, this inequality holds for sufficiently small $\Theta$. Once $\Theta$ is chosen, we have to choose the bound
$U$ for the control at the second stage in accordance with Inequality (\ref{condA}). Then Conditions A and B are met. The
switching to the third, terminal stage should happen upon arriving at the boundary
$\{(\mathfrak{Q}\delta(\Theta){\mathfrak x},\delta(\Theta){\mathfrak x})=1\}$ of $G_\Theta$. Here the vector ${\mathfrak
x}$ is related to the phase vector $x$ by $x=D{\mathfrak x}$ and $D$ is matrix (\ref{D}).

The switching to the second stage of control, when the bound for admissible controls drops from 1 to $U$, should happen
before getting into the inefficiency zone of the initial control. Therefore, the switching should happen upon reaching
the value $C(\underline\omega)$ of the  ``radius''.

\section{Final asymptotic result}\label{asymp}
Now we can state the final asymptotic theorem:

\begin{theorem}\label{main_approx22}
Assume that system (\ref{syst1})--(\ref{syst2}) of oscillators is non-resonant. Let $T=T(x)$ be the motion time from the
initial point $x$ to the equilibrium under our three-stage control, and let $\tau=\tau(x)$ be the minimum time. Then, as
$\rho(x)\to+\infty$, we have asymptotic equalities
\begin{equation}\label{approx_T12}
    \rho(x)/T(x)=1+o(1), \, \tau(x)/T(x)=1+o(1).
\end{equation}
In the resonant case, we have non-asymptotic inequalities
\begin{equation}\label{approx_T121}
    C(\underline\omega)\geq\rho(x)/T(x)\geq c(\underline\omega), \, 1\geq\tau(x)/T(x)\geq c(\underline\omega)
\end{equation}
for $\rho(x)\geq1,$ where $C(\underline\omega),\,c(\underline\omega)$ are strictly positive constants, depending on
eigenfrequencies of the system.
\end{theorem}

\begin{proof}
The proof is accomplished by relying upon the already proved results. Consider first the controlled motion from the value
$\rho(x)$ of the ``radius'' to  the value $\sqrt{\rho(x)}$. It follows from Theorem \ref{main_approx} that in the
non-resonant case the time spent under control (\ref{approx4}) is asymptotically equivalent to
$\rho(x)-\sqrt{\rho(x)}\sim\rho(x)$ as  $\rho(x)\to+\infty$, while for any other control, including the time-optimal one,
the time spent  is no less asymptotically. Then we move to the boundary of the inefficiency zone. It is clear, in view of
Theorem \ref{observation5}, that the motion time under control (\ref{approx4}) is $O(\sqrt{\rho(x)})$, which is
negligible compared to $\rho(x)$. The remaining two stages of the motion to zero, according to Theorems
\ref{observation52} and \ref{main} take a (uniform over all initial conditions) finite time. Therefore, they are
negligible and the total duration is asymptotically $\rho(x)$, while the optimal time is asymptotically the same.

To prove inequalities (\ref{approx_T121}) one could argue in the same way, by appealing to Theorem \ref{observation5}
instead of Theorem \ref{main_approx}.
\end{proof}

\section{Toy model: $N=1$}\label{toy}
We illustrate our previous constructions in the simplest case of a single oscillator. For a further simplification, we
assume that it has the unit frequency, so that the control system is
\begin{equation}\label{n=12model}
\begin{array}{l}
    \dot{x}=y,\\
    \dot{y}=-x+{u},\,|u|\leq1.
\end{array}
\end{equation}
We divide the entire phase space $\mathbb{R}^2$ into three domains. The ``basic'' one is the exterior of the disk
$\mathbb{B}_2$ of radius 2, wherein we apply the ``dry-friction'' control $u=-\sign(y)$. In principle, one can use a disk
$\mathbb{B}_r$ of any radius $r>1$. A substantially different control
\begin{equation}\label{n=12control}
    u(x,y)=x+6\mathfrak{T}^{-2}x-3\mathfrak{T}^{-1}y
\end{equation}
is applied in a neighborhood of zero. Here $\mathfrak{T}$ is the function of $(x,y)$ defined by equation (\ref{condu}),
where $\kappa^2=1/{\mathfrak Q_{11}}=1/6$. In this case it takes the form
\begin{equation}\label{conduu}
    \mathfrak{T}^{-2}6y^2-\mathfrak{T}^{-3}24xy+\mathfrak{T}^{-4}36x^2=1/6.
\end{equation}
The neighborhood $G_\Theta$ of zero in which this control is used, is the interior of the ellipse
\begin{equation}
\Theta^{-2}6y^2-\Theta^{-3}24xy+\Theta^{-4}36x^2=1, 
\end{equation}
where the parameter $\Theta=3^{1/4}$ is found from (\ref{condB}).
The ellipse contains the disk $\mathbb{B}_{\Lambda}$ of radius $\Lambda=(\lambda_{\rm max})^{-1/2}=0.26253\dots$, where
$\lambda_{\rm max}$ is the largest eigenvalue of the matrix of quadratic form (\ref{conduu}). The complete description of
control is as follows: in $\mathbb{R}^2\setminus(\mathbb{B}_2\cup G_\Theta)$ we apply the control $u=-\sign(y),$ in
$\mathbb{B}_2\setminus G_\Theta$ the control $u=-U\sign(y),$ where $U=\Lambda/2,$ finally, in $G_\Theta$ we apply control
(\ref{n=12control}), where $\mathfrak{T}$ satisfies (\ref{conduu}). If we would use the disk $\mathbb{B}_r,\, r>1$,
instead of $\mathbb{B}_2$ at the first stage, the parameter $U$ would be $\Lambda/r.$

\section*{Acknowledgements}

We are grateful to Alexander Barg and Nickolay Bolotnik for numerous suggestions, which help us  improve the presentation
of our work. This work was supported by the Russian Foundation for Basic Research (grants 11-08-00435, 14-08-00606 and
14-01-00476) and the Dynasty Foundation.

\newpage

{\setcounter{equation}{0} \setcounter{theorem}{0}\setcounter{lemma}{0} \setcounter{corollary}{0}
\def\theequation{A.\arabic{equation}}\def\thetheorem{A.\arabic{theorem}} \def\thelemma{A.\arabic{lemma}}
\def\thecorollary{A.\arabic{corollary}}
\newcounter{append}
\setcounter{append}{0}
\def\theappend{\Roman{append}}\refstepcounter{append}
\section*{APPENDIX \theappend. Asymptotics of the support function $H_{\mathcal{D}(T)}$}\label{support}
We present here a sketch of the proof of Theorem \ref{support0}.

By definition, $H_{\mathcal{D}(T)}(p)=\sup\langle x(T),p\rangle$, where $\sup$ is taken over admissible controls, and
$x(T)$ is the state at time $T$ of the control system (\ref{syst1})--(\ref{syst2}) such that $x(0)=0$. In view of the
Cauchy formula,
\begin{equation}
\langle x(T),p\rangle=\int_0^T \langle e^{A(T-t)}Bu(t),p\rangle dt=\int_0^T u(t)B^*e^{A^*(T-t)}p dt,
\end{equation}
and upon taking the supremum under the integral sign and performing a change of variables $t\mapsto T-t$, we obtain
\begin{equation}
H_{\mathcal{D}(T)}(p)=\int_0^T \sup_{|u(t)|\leq1} u(t)B^*e^{A^*(T-t)}p dt=\int_0^T |B^*e^{A^*t}p| dt.
\end{equation}
In coordinates $\xi_i,\eta_i, $ the last formula takes the form
\begin{equation}
\label{spt}H_{\mathcal{D}(T)}(p)=\int_0^T \left|\sum_{i=1}^N\eta_i\cos\omega_it+\omega_i^{-1}\xi_i\sin\omega_it\right|
dt.
\end{equation}
This expression represents an integral of the function
\begin{equation}
f(\varphi)=\left|\sum_{i=1}^N\eta_i\cos\varphi_i+\omega_i^{-1}\xi_i\sin\varphi_i\right|
\end{equation}
taken over the rectilinear winding $\varphi_i(t)=\omega_it$ of the torus $\mathcal{T}=(\mathbb{R}/2\pi\mathbb{Z})^N$ with
angular coordinates $\varphi_i$.
 Suppose that the system of oscillators is nonresonant, i.e., condition (\ref{resonance}) is fulfilled.
Then \cite{Arnold}, the time average $\lim\limits_{T\to\infty}\frac1T\int_0^T f(\varphi(t)) dt$ coincides with the space
average $\int_{\mathcal{T}}f(\varphi)d\varphi$. In order to prove Theorem \ref{support0}, we note that
\begin{equation}
\eta_i\cos\varphi_i+\omega_i^{-1}\xi_i\sin\varphi_i=z_i\cos(\varphi_i+\alpha_i), 
\end{equation}
where $\alpha=(\alpha_i)$ is a
constant point of the torus. Therefore
\begin{equation}
\int_{\mathcal{T}}f(\varphi)d\varphi=\int_{\mathcal{T}}f(\varphi-\alpha)d\varphi=\int_{\mathcal{T}}\left|\sum_{i=1}^N
z_i\cos\varphi_i\right|d\varphi.
\end{equation}
Thus,
\begin{equation}
\lim\limits_{T\to\infty}\frac1T H_{\mathcal{D}(T)}(p)=\int_{\mathcal{T}}\left|\sum_{i=1}^N
z_i\cos\varphi_i\right|d\varphi
\end{equation}
which is the claim of Theorem \ref{support0}.

\refstepcounter{append}
\section*{APPENDIX \theappend. Elliptic integrals}\label{apelliptic}
Here we study our basic function (\ref{H}) in the case $N=2$ when it belongs to the realm of elliptic functions. In this
case,
\begin{equation}\label{ellptic1}
    \frac{\partial {\mathfrak H}}{\partial z_i}=\frac{1}{(2\pi)^2}\iint \cos\varphi_i \sign(z_{1}\cos\varphi_{1}+z_{2}
    \cos\varphi_{2})d\varphi_1 d\varphi_2.
\end{equation}
To fix ideas, consider the case $i=1$ and perform the inner integration over $\varphi_{2}$. Taking positivity of $z_2$
into account, we have to compute the integral
\begin{equation}\label{ellptic3}
    \frac{1}{2\pi} \int_0^{2\pi}
    \sign(-C+\cos\varphi_{2})d\varphi_{2}=\frac{2}{\pi}\arccos C-1, \mbox{ where }|C|\leq1,
\end{equation}
where $C=k\cos\phi_1,$ $k=-z_1/z_2$. One can assume, by making an interchange of the indices if necessary, that $|k| \leq
1$. We note that this assumption introduces a  ``disparity'' between $z_1$ and $z_2$. From equation (\ref{ellptic3}), we
obtain that if $|k|\leq1$, then
\begin{equation}\label{ellptic4}
    \frac{\partial {\mathfrak H}}{\partial z_1}=\frac{1}{\pi^2}\int_0^{2\pi}\cos\varphi_1\arccos (k\cos\varphi_1)d\varphi_1,
\end{equation}
since $\int_0^{2\pi}\cos\varphi_1 d\varphi_1=0$. Integrating by parts, we can rewrite the integral in (\ref{ellptic4}) in
an ``elliptic'' form:
\begin{equation}\label{ellptic5}
\int_0^{2\pi}\cos\varphi\arccos (k\cos\varphi)
    d\varphi=\int_0^{2\pi}\frac{k\sin^2\varphi}{\sqrt{1-k^2\cos^2\varphi}}\,
    d\varphi.
\end{equation}
This gives the final formula for the derivative of the support function
\begin{equation}\label{ellptic6}
    \frac{\partial {\mathfrak H}}{\partial z_1}=\frac{1}{\pi^2}\int_0^{2\pi}\frac{k\sin^2\varphi}{\sqrt{1-k^2\cos^2\varphi}}\, d\varphi, \mbox{ where  }k=-z_1/z_2,
\end{equation}
valid for  $|k|\leq1$. To compute $\frac{\partial {\mathfrak H}}{\partial z_2}$, we need the inner integral
\begin{equation}\label{ellptic31}
    \frac{1}{2\pi} \int_0^{2\pi}
    \cos\varphi_{2}\sign(-C+\cos\varphi_{2})d\varphi_{2}=\frac{2}{\pi}\sin\arccos C, \mbox{ if }|C|\leq1,
\end{equation}
which gives
\begin{equation}\label{ellptic7}
    \frac{\partial {\mathfrak H}}{\partial z_2}=\frac{1}{\pi^2}\int_0^{2\pi}{\sqrt{1-k^2\cos^2\varphi}}\, d\varphi.
\end{equation}
Note that the apparent asymmetry between the integral formulas (\ref{ellptic6}) and (\ref{ellptic7}) is misleading: the
change of variables $z_1\leftrightarrows z_2$ implies the change of parameters $k\leftrightarrows k^{-1}$. Under this
change, the integrals
\begin{equation}
    I_1(k)=\int_0^{2\pi}\frac{k\sin^2\varphi}{\sqrt{1-k^2\cos^2\varphi}}\, d\varphi\mbox{ and }I_2(k)=\int_0^{2\pi}{\sqrt{1-k^2\cos^2\varphi}}\, d\varphi,
\end{equation}
regarded as (multivalued) meromorphic functions of  $k$, are transposed: $I_1(k^{-1})=I_2(k)$. The functions $I_i$ are
integrals of a meromorphic differential form
\begin{equation}
    \alpha=\frac{(1-k^2x^2)dx}{y} \mbox{ on the elliptic curve } \mathcal{E}=\{y^2=(1-x^2)(1-k^2x^2)\},
\end{equation}
taken over some pathes $\gamma_i,$ where $\gamma_1$ goes from $(-1,0)$ to $(1,0)$ and gets back, while $\gamma_2$ goes
from $(-k^{-1},0)$ to $(k^{-1},0)$ and gets back. The form $\alpha$ has a second order pole at infinity, so that it is a
differential of the second kind. The key equation (\ref{z}) that defines control (\ref{control}), has the form of
equation for $k=-z_1/z_2$:
\begin{equation}\label{ellptic_control}
    \frac{e_2}{e_1}=\frac{I_2}{I_1}{(k)}={\int_{\gamma_2}\alpha}\left\slash{\int_{\gamma_1}\alpha}\right..
\end{equation}
We note that the support function itself has the form
\begin{equation}\label{ellptic_integral}
\mathfrak{H}(z_1,z_2)=\frac{1}{\pi^2}\int_0^{2\pi}\frac{(z_2^2-z_1^2)d\varphi}{\sqrt{z_2^2-z_1^2\cos^2\varphi}} \mbox{ if
}|z_1|\leq|z_2|,
\end{equation}
and is expressed via a period of the holomorphic form $\frac{dx}{y}$ on $\mathcal{E}$.

\refstepcounter{append}
\section*{APPENDIX \theappend. Another representation of the function $\mathfrak H(z)$}\label{gothic_H}
Besides Definition (\ref{approx2N}), there is another useful representation \cite{ovseev} of the hypergeometric function
$\mathfrak {H}(z)$. Namely,
\begin{equation}\label{bessel}
\mathfrak {H}(z)=\frac1{\pi}\int_0^\infty \left({1-\prod_{i=1}^N J_0(z_i\lambda)}\right)\frac{d\lambda}{\lambda^2},
\end{equation}
where
\begin{equation}
J_0(x)=\frac{1}{\pi}\int\limits_0^{\pi}e^{ix\cos\phi}\,d\phi=\sum_{k=0}^\infty
\frac{(-1)^k}{k\,!^2}{\left(\frac{x}{2}\right)}^{2k}
\end{equation}
is the Bessel function of order zero. For any real $x$ we have
\begin{equation}\label{abs}
|x|=\frac1{2\pi}\int_{\mathbb R}\frac{1+i\lambda x -e^{i\lambda x}}{\lambda^2}\,{d\lambda},
\end{equation}
where the integral  is to be understood as $\lim_{T\to\infty}\int_{-T}^T\dots d\lambda$. Indeed, the RHS of $I(x)$ has
the property $I(\mu x)=|\mu|I(x)$ for any real $x$.   This argument proves (\ref{abs}) up to a constant factor. To
determine this factor, we consider the second (distibutional) derivative of the RHS and LHS of (\ref{abs}). This reduces
the problem to  the identity
\begin{equation}\label{abs2}
\delta(x)=\frac1{2\pi}\int_{\mathbb R}e^{i\lambda x}\,{d\lambda}
\end{equation} or, equivalently,
\begin{equation}\label{abs3}
\phi(0)=\frac1{2\pi}\int_{\mathbb R}\int_{\mathbb R}\phi(x)e^{i\lambda x}\,dx\,{d\lambda},
\end{equation}
where $\phi$ is a Schwartz function,  which is a well known formula for the  inverse Fourier transform. Therefore,
\begin{eqnarray}\label{abs22}
\int\left|\sum_{k=1}^N z_k\cos\varphi_k\right|d\varphi=\frac1{2\pi}\int_{\mathcal T}\int_{\mathbb R}\frac{e^{i\lambda
\sum_{k=1}^N z_k\cos\varphi_k}-i\lambda \sum_{k=1}^N z_k\cos\varphi_k-1}{\lambda^2}\,{d\lambda}d\varphi\nonumber\\
\nonumber =\frac1{2\pi}\int_{\mathcal T}\int_{\mathbb R}\frac{1-\prod_k e^{i\lambda
z_k\cos\varphi_k}}{\lambda^2}\,{d\lambda}d\varphi=\frac1{\pi}\int_0^\infty \left({1-\prod_{i=1}^N
J_0(z_i\lambda)}\right)\frac{d\lambda}{\lambda^2}.
\end{eqnarray}

\refstepcounter{append}
\section*{APPENDIX \theappend. Differentiability properties of  functions $\mathfrak H,H,\mathfrak R,\rho$}\label{rho}

Here we study basic analytic properties of the integral \begin{equation}\label{H}\mathfrak
H(z)=\int_{\mathcal{T}}\left|\sum_{i=1}^N z_i\cos\varphi_i\right|d\varphi\end{equation} as a function of $z\in\mathbb
R^N$ and derive  differentiability properties  of  functions $H,\rho$, and $\mathfrak{R}$.

First, it is clear that $\mathfrak H(z)$ is of class $C^1$ outside zero, and
\begin{equation}\label{gradient}\frac{\partial {\mathfrak H}}{\partial z_i}=\int_{\mathcal{T}}\sign\left(\sum_{i=1}^N
z_i\cos\varphi_i\right)\cos\varphi_i\,d\varphi,\end{equation} because the integrand in (\ref{gradient}) is bounded and
continuous with respect to $z$ outside the analytic hypersurface
\begin{equation}\label{zero}V(z)=\left\{\varphi\in\mathcal{T}:f(z,\varphi)=0\right\},\quad f(z,\varphi)=\sum_{i=1}^N
z_i\cos\varphi_i\end{equation} of $d\varphi$-measure zero (cf. \cite{federer} \S 3.1). As for the second derivatives, we
again have the integral formula
\begin{equation}\label{hessian2}
    \left\langle\frac{\partial^2 {\mathfrak H}}{\partial z^2}(z)\xi,\xi\right\rangle=\int\limits_{V(z)}\left(\sum_{i=1}^N\xi_{i}\cos\varphi_{i}\right)^2d\sigma(\varphi),
\end{equation}
where
\begin{equation}
   d\sigma(\varphi)=\frac{d\varphi}{df}=\frac{d\varphi_1\wedge\dots\wedge d\varphi_N}{(2\pi)^N df}
\end{equation}
is the canonical volume element on $V(z)$. The problem is that the positive measure $d\sigma(\varphi)$ is not necessarily
finite: there are exceptional vectors $z$ such that the integral (\ref{hessian2}) is $+\infty$ for all vectors $\xi$ not
collinear with $z$. We proceed to determine the exceptional locus. It is convenient to make the substitution
$t_i=\cos\varphi_{i}$ and assume without loss of generality that $z_N\neq0$. The measure $d\sigma$ can be rewritten as
$f(t)dt_1 \dots dt_{N-1}$, where
\begin{equation}\label{dsigma}
f(t)=z_N^{-1}{(2\pi)^{-N} \prod_{i=1}^{N-1}(1-t_i^2)^{-1/2}\left(1-\frac1{z_N^2}\left(\sum_{i=1}^{N-1}
z_it_i\right)^2\right)^{-1/2}}
\end{equation}
on the polytope defined by conditions
\begin{equation}\label{polytope}|t_i|\leq1,\,i=1,\dots,N-1,\quad\left|\sum_{i=1}^{N-1} z_it_i\right|\leq z_N.\end{equation}
If the linear forms $1\pm t_i,\,i=1,\dots,N-1$ and $1\pm \frac1{z_N}\left(\sum_{i=1}^{N-1} z_it_i\right)$ are all
different, then $f$ is Lebesgue-integrable. The opposite  happens exactly when
\begin{equation}\label{singular}
z_i=\pm z_j,\, i\neq j, \mbox{ and } z_k=0 \mbox{ for } k\neq i,j \mbox{ and } i,j,k=1,\dots,N.
\end{equation}
Then the singularity takes the nonintegrable form $(1\pm t_i)^{-1}$.  Thus, condition (\ref{singular}) determines the
exceptional locus $\sing ({\mathfrak H})$, where the quadratic form $\frac{\partial^2 {\mathfrak H}}{\partial
z^2}(z)=+\infty$ on the quotient space $\mathbb R^N/\mathbb R z$. The corresponding locus $\sing ({\mathfrak R})$ for the
dual function $\mathfrak R$ can be obtained from the set (\ref{singular}) by the gradient map $\psi(z)=\frac{\partial
{\mathfrak H}}{\partial z}(z)$.

More precisely, $\sing ({\mathfrak R})$ is the set of points $\rho\frac{\partial {\mathfrak H}}{\partial z}(z)$, where
$z\in \sing ({\mathfrak H})$  and $\rho$ is an arbitrary positive factor. Formula (\ref{gradient}) implies immediately
that $\psi$ maps the exceptional locus (\ref{singular}) into itself. Luckily, it turns out that $\frac{\partial^2
{\mathfrak R}}{\partial z^2}(e)$ is continuous everywhere outside zero so that there is no exceptional set for the dual
function. The reason is simple: $\frac{\partial^2 {\mathfrak H}}{\partial z^2}(z)=+\infty$ at singular points  which
means that $\frac{\partial^2 {\mathfrak R}}{\partial z^2}(e)=0$ at the corresponding point $e$. Indeed, this follows from
the general duality relation (cf. (\ref{rho02}))
\begin{equation}\label{relation}
1={\mathfrak R}\frac{\partial^2 {\mathfrak H}}{\partial z^2}\frac{\partial^2{\mathfrak R}}{\partial
e^2}+\frac{\partial{\mathfrak R}}{\partial e}\otimes\frac{\partial {\mathfrak H}}{\partial
    z}.
\end{equation}

An important observation is this: Consider the canonical map $\pi$ from the space of quadratic forms on $\mathbb
R^N/\mathbb R z$ of dimension $N(N-1)/2$ to the corresponding sphere $S^{N(N-1)/2-1}$ of rays. This map establishes a
correspondence between a quadratic form and  all its multiples by a positive factor. Then the map
\begin{equation}\label{map}
    z\mapsto \pi \frac{\partial^2 {\mathfrak H}}{\partial z^2}(z)
\end{equation}
is continuous. Indeed, at the singular locus the non-integrability of the measure $d\sigma$ affects the RHS of Identity
(\ref{hessian2}) like multiplication by an infinite positive scalar factor. In particular, this means that the ratio
\begin{equation}
    \left\langle\eta,\frac{\partial^2 {\mathfrak H}}{\partial z^2}(z)\zeta\right\rangle:\left\langle\zeta,\frac{\partial^2 {\mathfrak H}}{\partial z^2}(z)\zeta\right\rangle,
\end{equation}
is a continuous function of $z$. Here $\eta$ and $\zeta$ are continuous vector fields in $\mathbb R^N$ and $\zeta(z)$ is
not collinear with $z$.

The duality relation (\ref{relation}) allows one to draw a similar conclusion for the quadratic form
$\frac{\partial^2{\mathfrak R}}{\partial e^2}$.

Now we turn to singularities of the second derivatives of the dual pair of functions $H(p)={\mathfrak H}(z(p))$ and
$\rho(x)={\mathfrak R}(e(x))$. The corresponding singular locus $\sing (H)$ can include singular points of the mapping
$z:{\mathbb R}^{2N}\to {\mathbb R}^{N}$ outside the preimage $z^{-1}(\sing ({\mathfrak H}))$. A direct computation gives
the relation
\begin{equation}\label{reduction}
    \frac{\partial^2 {H}}{\partial z^2}=\frac{\partial{ z}}{\partial p}^*\frac{\partial^2 {\mathfrak H}}{\partial z^2}\frac{\partial{ z}}{\partial p}+\frac{\partial {\mathfrak H}}{\partial z}\frac{\partial^2{ z}}{\partial p^2},
\end{equation}
and from the identity $z_i=\langle Q_ip,p\rangle^{1/2}$ for a nonnegative symmetric matrix $Q_i$, we obtain that
\begin{equation}\label{derivative}
    \frac{\partial^2{ z_i}}{\partial p^2}=\frac1{z_i}\left(Q_i-\frac{Q_ip\otimes Q_ip}{z_i^2}\right).
\end{equation}
The above expression is clearly singular as $z_i\to0$ but $z_i\frac{\partial^2{ z_i}}{\partial p^2}$ is bounded (and
nonnegative). The matrix $\frac{\partial{ z}}{\partial p}$ is everywhere bounded.  Thus, in order to find singularities
of $\frac{\partial^2 {H}}{\partial z^2}$ we have to find
\begin{equation}\label{derivative2}
\lim\limits_{z_i\to0}\frac{1}{z_i}\frac{\partial {\mathfrak H}}{\partial z_i}(z).\end{equation} It is clear from
(\ref{gradient}) that $\frac{\partial {\mathfrak H}}{\partial z_i}(z)=0$ if the component $z_i=0$ because the
one-dimensional integral $\int_0^{2\pi}\cos\varphi_i\, d\varphi_i=0$. Therefore, Expression (\ref{derivative2}) equals
$\frac{\partial^2 {\mathfrak H}}{\partial z_i^2}$, and
\begin{equation}\label{derivative3}
\frac{\partial {\mathfrak H}}{\partial z}\frac{\partial^2{ z}}{\partial p^2}=\sum_{i=1}^N\frac{1}{z_i}\frac{\partial
{\mathfrak H}}{\partial z_i}\left(Q_i-\frac{Q_ip\otimes Q_ip}{z_i^2}\right)\end{equation} tends to
\begin{equation}\sum_{i=1}^N\frac{\partial^2 {\mathfrak H}}{\partial z_i^2}\left(Q_i-\frac{Q_ip\otimes Q_ip}{z_i^2}\right).\end{equation} The
last expression is a nonnegative symmetric matrix because of inequality $\frac{\partial^2 {\mathfrak H}}{\partial
z_i^2}\geq0$, implied by the convexity of ${\mathfrak H}$, and because of the Cauchy inequality. The term $
\frac{\partial{ z}}{\partial p}^*\frac{\partial^2 {\mathfrak H}}{\partial z^2}\frac{\partial{ z}}{\partial p}$ from
(\ref{reduction}) defines a strictly positive quadratic form on $\mathbb R^{2N}/\mathbb R p$.
 Therefore, outside the preimage $z^{-1}(\sing ({\mathfrak H}))$ the symmetric
matrix $\frac{\partial^2 {\mathfrak H}}{\partial z^2}(z)$ remains locally bounded and strictly positive, although it is
not continuous at points $p$, where a component $z_i(p)=0$. In view of duality relation (cf. (\ref{rho02})),
\begin{equation}\label{relation2}
1={\rho}\frac{\partial^2 {H}}{\partial p^2}\frac{\partial^2{\rho}}{\partial x^2}+\frac{\partial{\rho}}{\partial
x}\otimes\frac{\partial {H}}{\partial p}
\end{equation}
we conclude that the symmetric matrix $\frac{\partial^2 {\rho}}{\partial x^2}(x)$ is  bounded on the ``sphere''
\begin{equation}
\omega=\{x\in\mathbb R^{2N}:\rho(x)=1\},
\end{equation}
but it is discontinuous at points $x$ such that a component $e_i(x)=0$.

\refstepcounter{append}
\section*{APPENDIX \theappend. Perturbation theory of observable linear systems} \label{perturb_observ}
The subject of the Kalman observability theory is a linear time-invariant system $\dot x=\alpha x$, which is observed, so
that the vector $y=\beta x$ is the observation result. Here $\alpha$ and $\beta$ are constant matrices. The system is
said to be completely observable, if the knowledge of the curve $y(t)$ in an open time interval allows to recover $x(t)$
uniquely. We consider a perturbed situation where the observed vector has the same structure, but the vector $x$
satisfies the perturbed equation $\dot x=\alpha x+f$. Then, it is impossible to recover $x$ from $y$ precisely, but if
the perturbation $f$ is small, we can do this with a small error.

In quantitive terms, the error size is described by the following theorem.
\begin{theorem}\label{observation0}
Suppose that $\dot x=\alpha x,\,y=\beta x$ is a completely observable time-invariant linear system. The following a
priori estimate holds for a solution $z$ of $\dot z=\alpha z+f$ in the interval $I$ of {\bf integer} length:
\begin{equation}\label{observ00}
    \int_I|z|dt\leq C\left(\int_I|\beta z|dt+\int_I\left|f\right|dt\right),
\end{equation}
where the constant $C$ does not depend on the interval $I$.
\end{theorem}

The proof of Theorem \ref{observation0} is based on the following Lemma:
\begin{lemma}\label{observation11}
Under the assumptions of Theorem \ref{observation0} consider the map
\begin{equation}
     \Phi:z\mapsto [y,f]=[\mathcal{C}z,\dot z-\mathcal{A}z]
\end{equation}
from $\mathbb{W}=W^{1,1}\otimes\mathbb{R}^n$ to
$\mathbb{L}=\mathcal{L}_1\otimes\mathbb{R}^m\oplus\mathcal{L}_1\otimes\mathbb{R}^n$ and its image $L= \Phi(\mathbb{W})$.
Then the image $L$ of the map $\Phi$  is closed in $\mathbb{L}$.
\end{lemma}
Here $W^{n,1}$ is the Sobolev space of functions with $n$ integrable derivatives.
\begin{proof}
We consider the subspace $M\subset L$ formed by vectors $\Phi(z)$ such that the function $z$ vanishes at 0: $z(0)=0$.
This is a closed subspace of $\mathbb{L}$, because the map $z\mapsto f=\dot z-\mathcal{A}z$ defines an isomorphism
$M\!\backsimeq\mathcal{L}_1\otimes\mathbb{R}^n$. Indeed, the Cauchy problem,
\begin{equation}
    \dot{z}=\mathcal{A}z+f,\quad z(0)=0,
\end{equation}
is correctly solvable. Another important subspace of $N\subset L$ is formed by  vectors $\Phi(z)$ such that $\dot
z-\mathcal{A}z=0$. It is also closed in  $\mathbb{L}$, because it it is finite dimensional $(\dim N=n)$. Since $L$ is a
direct sum of $M$ and $N$, it is  closed in $\mathbb{L}$.
\end{proof}

It is easy to derive Theorem \ref{observation0} from Lemma \ref{observation11}: The map $\Phi:\mathbb{W}\to L$ is a
continuous linear map. By Lemma \ref{observation11} the image $L$ is closed in $\mathbb{L}$. The observability condition
means that the kernel of the map $\Phi$ is zero. Hence, one can apply the Banach inverse operator theorem and conclude
that
\begin{equation}\label{observ0}
    |z|_1\leq c(|\mathcal{C}z|_0+|\dot z-\mathcal{A}z|_0).
\end{equation}
Here $c$ is the norm of the inverse operator $\Phi^{-1}$, and
\begin{equation}
    |z|_n=\sum_{k=0}^n\int_0^1\left|\frac{\partial^k z}{\partial t^k}\right|dt
\end{equation}
is the standard Sobolev norm in $W^{n,1}([0,1])$. The conclusion of Theorem \ref{observation0} is an obvious relaxation
of inequality (\ref{observ0}).

\refstepcounter{append}
\section*{APPENDIX \theappend. Proof of Lemma \ref{canonical}}\label{brunovsky_lemma}

\begin{proof}
We begin with identity (\ref{C}). The feedback matrix $C$ can be found from the condition of the nilpotency of the matrix
$A+BC$. In other words, we require that the characteristic polynomial $P(s)=\det(s-(A+BC))$ be equal to $s^{2N}$. We
rewrite $P(s)$ in the form $\det\left((s-A)(1-(s-A)^{-1}BC)\right)$ and use the general property of determinants
\cite{kailath}:
\begin{equation}\label{detCommut}
    \det(1_n-\alpha\beta)=\det(1_m-\beta\alpha)
\end{equation}
for any pair $\alpha,\beta$ of matrices of size $n\times m$ and $m\times n$, respectively. By applying (\ref{detCommut})
to the pair
\begin{equation}
    \alpha=(s-A)^{-1}B,\quad \beta=C
\end{equation}
we obtain that
\begin{equation}\label{det2}
    P(s)=\det(s-A)-CF(s,A)B,
\end{equation}
where $F(s,A)=[\det(s-A)](s-A)^{-1}$. Note that elements of $F(s,A)$ are polynomials of degree less than $2N$ in $s$,
because for they are cofactors to some elements of the matrix $(s-A)$. Then $CF(s,A)B$ is a {\em scalar} polynomial with
the same bound for the degree. If the matrix $C$ is given by (\ref{C}), then $CF(s,A)B$ has the form $\sum
{c_k}{\prod_{i\neq k}(s^2+\omega_i^2)},$ and $\det(s-A)=\prod_{i=1}^N(s^2+\omega_i^2)$. Therefore, equation $P(s)=s^{2N}$
is equivalent to the following identity:
\begin{equation}\label{interpol}
    \prod_{i=1}^N(s^2+\omega_i^2)-s^{2N}=\sum {c_k}{\prod_{i\neq k}(s^2+\omega_i^2)}.
\end{equation}
This is the Lagrange interpolation formula for the polynomial $f(\lambda)=\prod_{i=1}^N(\omega_i^2+\lambda)-\lambda^N$ of
degree $N-1$ with nodes $\lambda=-\omega_i^2,\,i=1,\dots,N$, which implies (\ref{C}).

We prove statements (\ref{AB}) and (\ref{e}) simultaneously. We already know that the matrix ${\widetilde A}=A+BC$ is
nilpotent: ${\widetilde A}^{2N}=0$. Define a new basis by formula (\ref{e}):
$\mathfrak{e}_i=\frac{(-1)^{i-1}}{(i-1)!}{\widetilde A}^{i-1}B$ for $i=1,\dots,2N.$ The fact that the vectors
$\mathfrak{e}_i$ form a basis follows from the complete controllability of system (\ref{syst1})--(\ref{syst2}). It is
clear that ${e}_1=B$ and ${\widetilde A}\mathfrak{e}_i=-i \mathfrak{e}_{i+1}$ for $i<2N$. For $i=2N$ it follows from the
nilpotency of ${\widetilde A}$ that ${\widetilde A}\mathfrak{e}_{2N}=\frac{(-1)^{2N-1}}{(2N-1)!}{\widetilde A}^{2N}B=0$.
This shows that the matrix ${\widetilde A}$ has canonical form (\ref{AB}) in the basis (\ref{e}).

We show now that the matrix $D$ can be represented as block-matrix (\ref{d}). The vectors $\mathfrak{e}_{i}$ are, by
definition, the columns of $D$. Denote by $\lambda$ and $\omega^2$ the diagonal matrices
\begin{equation}
    \lambda=\diag(\lambda_1,\lambda_1,\dots,\lambda_N,\lambda_N) ,\quad
    \omega^2=\diag(\omega^2_1,\omega^2_1,\dots,\omega^2_N,\omega^2_N),
\end{equation}
where the scalar $\lambda_k$ is defined in (\ref{d}). It is obvious that $CB=0.$ Denote ${\widetilde A}B=AB$ by $B'$. It
is clear that $\mathfrak{e}_{1}=B$, and $\mathfrak{e}_{2}=-B'$. We compute $CB'=\sum_{i=1}^N c_i,$ where $c_i$ is defined
in (\ref{C}). We show that $\sum_{i=1}^N c_i=\sum_{i=1}^N \omega^2_i$. To do this, we divide both sides of
(\ref{interpol}) by $s^{2N-2}$ and pass to the limit $s\to\infty$. We get $\sum_{i=1}^N c_i$ in the RHS, and
$\sum_{i=1}^N \omega^2_i$ in the left-hand side. Now we can compute ${\widetilde
A}B'=A^2B+BCB'=-\omega^2B+\left(\sum_{i=1}^N \omega^2_i\right)B=\lambda B$ and ${\widetilde A}^2B'={\widetilde A}\lambda
B=\lambda{\widetilde A} B=\lambda B'$. Therefore, we conclude by induction that
\begin{equation}
    \mathfrak{e}_{2k-1}=\frac{(-1)^{k-1}}{(2(k-1))!}\lambda^{k-1}B \mbox{ and }\mathfrak{e}_{2k}=-\frac{(-1)^{k-1}}{(2k-1)!}\lambda^{k-1}B',
\end{equation}
which is equivalent to the block representation (\ref{e})--(\ref{D}) of the matrix $D$.
\end{proof}

\refstepcounter{append}
\section*{APPENDIX \theappend. Proof of Theorem \ref{main2}}\label{proof_2}

\begin{proof}
Consider orthogonal polynomials (shifted Jacobi polynomials) with respect to the measure $d\mu=(1-x)dx$ in the interval
$[0,1]$. The required polynomials $P_n$ are given by the Rodrigues formula
\begin{equation}\label{rodrigue}
    {P_n}(x) = \frac{1}{n!(1-x)}{\partial}^n  \left[(1-x) (x-x^2)^n \right],
\end{equation}
where ${\partial}=\frac{d}{dx}.$ Indeed, $\int {P_n}(x)x^m d\mu=0$ for $m<n$ since
\begin{equation}
\begin{array}{l}
    \int {P_n}(x)x^m d\mu=\frac1{n!}\int {\partial}^n \left[(1-x) (x-x^2)^n \right]x^m dx= \\[1em]
    \frac{(-1)^n}{n!}\int \left[(1-x) (x-x^2)^n \right]{\partial}^n x^m dx=0, \\
\end{array}
\end{equation}
where we used the identity ${\partial}^nx^m=0$ and integration by parts. Therefore, the polynomials $P_n$ and $P_m$ are
orthogonal if $n\neq m$. One can easily compute the leading coefficient $c_n$ of $P_n$. It is the same as the leading
coefficient of the polynomial ${\pi_n}(x) = \frac{(-1)^n}{n!x} {\partial}^n \left[x^{2n+1} \right]$, which obviously
equals $\frac{(-1)^n(2n+1)!}{n!(n+1)!}$. The square norm of the polynomial $P_n$ is
\begin{equation}
\begin{array}{l}
    \int {P_n}^2 d\mu=c_n\int {P_n}(x)x^n(1-x)dx= \\[1em]
    \frac{(-1)^n(2n+1)!}{n!(n+1)!}(-1)^n\int_0^1
    \left[x^n(1-x)^{n+1}\right]dx=\frac{(2n+1)!}{n!(n+1)!}B(n+2,n+1), \\
\end{array}
\end{equation}
where $B(\alpha,\beta)=\frac{\Gamma(\alpha)\Gamma(\beta)}{\Gamma(\alpha+\beta)}$ is the Euler $B$-function. Finally, we
have
\begin{equation}
    \int {P_n}^2 d\mu=\frac{(2n+1)!}{n!(n+1)!}\frac{\Gamma(n+2)\Gamma(n+1)}{\Gamma(2n+3)}
    =\frac{(2n+1)!}{n!(n+1)!}\frac{(n+1)!n!}{(2n+2)!}=\frac1{2(n+1)}.
\end{equation}
It follows immediately from the Rodrigues formula (\ref{rodrigue}) that $P_n\in\mathbb{Z}[x]$ because the operator
$\frac{1}{n!}{\partial}^n $ maps $\mathbb{Z}[x]$ into itself. This fact can be rewritten in the form
$P_{i-1}=\sum{a_{ij}m_j}$, where $m_j=x^{j-1}$ are elements of the standard monomial basis and $A=(a_{ij})$ is an integer
(triangular) matrix of coefficients of the Jacobi polynomials. The above formulas for the scalar product can be rewritten
in the form
\begin{equation}
    A\mathfrak{q} A^*=\diag\left(\frac1{2k}\right)_{k=1}^n,
\end{equation}
or, which is the same, in the form
\begin{equation}\label{finalQ}
    \mathfrak{Q}=A^*\diag\left(\frac{1}{2k}\right)A.
\end{equation}
The last formula obviously imply that $\mathfrak{Q}$ is an even integer matrix.
\end{proof}

\refstepcounter{append}
\section*{APPENDIX \theappend. Proof of Theorem \ref{main3}}\label{proof_3}

\begin{proof}
From (\ref{finalQ}), we obtain that
\begin{equation}
    \mathfrak{Q}_{11}=\sum_{k=1}^{2N}2ka_{k1}^2,
\end{equation}
where $a_{k1}=P_{k-1}(0)$ is the constant term of the Jacobi polynomial of degree $k-1$. This term is always 1 for the
following reason. It follows from the Rodrigues formula (\ref{rodrigue}) that
\begin{equation}
    {P_n}(0) = \frac{1}{n!}{\partial}^n \left[ (x-x^2)^n \right]\vert_{x=0}.
\end{equation}
But
\begin{equation}
    {\partial}^n \left[ (x-x^2)^n \right]\vert_{x=0}=(1-x)^n{\partial}^n \left[ x^n \right]\vert_{x=0}=n!.
\end{equation}
Therefore, $a_{k1}=1$ for all $k$, and
\begin{equation}
    \mathfrak{Q}_{11}=\sum_{k=1}^{2N}2k=2N(2N+1).
\end{equation}
\end{proof}

}
\end{document}